\documentclass[10pt,twoside,a4paper]{article}
\usepackage{amssymb}
\usepackage[all]{xy}
\usepackage{amssymb,amsfonts}
\usepackage{pb-diagram}
\usepackage{psfig}
\usepackage{amsmath,amsthm,amsfonts,amssymb}
\usepackage{graphicx}

\usepackage{soul}
\usepackage{hyperref}
\usepackage[usenames]{color}
\usepackage{colortbl}

\newcommand{\be}{\begin{enumerate}}
\newcommand{\ee}{\end{enumerate}}
\newcommand{\beq}{\begin{equation}}
\newcommand{\eeq}{\end{equation}}
\newcommand{\ep}{\varepsilon}

\newenvironment{blue}{\color{blue}}{}

\newtheorem{cor}{Corollary}
\newtheorem{proposition}{Proposition}
\newtheorem{theorem}{Theorem}
\newtheorem{lemma}{Lemma}
\newtheorem{definition}{Definition}
\newtheorem{remark}{Remark}

\begin{document}

\title{Equations and fully residually free groups}
\author{\textsf{Olga Kharlampovich and  Alexei Myasnikov}\\
Mini-course for the GCGTA conference in Dortmund (2007),\\
Ottawa-Saint Sauveur conference (2007), Escola
d'Algebra in Rio de Janeiro (2008) \\and Alagna (Italy, 2008) conference on equations in groups.
} \maketitle

\section{Introduction}

\subsection{Motivation}
Solving equations is one of the main themes  in mathematics. A large
part of the combinatorial group theory revolves around the word and
conjugacy problems - particular types of equations in groups.
Whether a given equation has a solution in a given group is, as a
rule,  a non-trivial  problem.  A more general and more difficult
problem is to decide  which formulas of the first-order logic hold in a given group.

Around 1945 A. Tarski put forward two problems on elementary theories of free groups that served as a motivation for much of the research in group theory and logic for the last sixty years. A joint  effort of mathematicians of several generations culminated in the following theorems, solving  these Tarski's conjectures.
\begin{theorem} [Kharlampovich and Myasnikov \cite{KM4}, Sela \cite{Sela6}] The
elementary theories  of non-abelian free groups coincide.
\end{theorem}

\begin{theorem} [Kharlampovich and Myasnikov \cite{KM4}] The elementary theory of a free group $F$ (with constants for elements from $F$ in the language) is decidable. \end{theorem}
 We recall
that the \emph{elementary theory} $Th(G)$ of a group $G$ is the set
of all first order sentences in the language of group theory which
are true in $G$. A discussion of these conjectures can be found in
several textbooks on logic, model and group theory (see, for example,  \cite{CK}, \cite{EP}, \cite{LS}).

 The work on the Tarski conjectures  was  rather fruitful - several areas of group theory were developed along the way. It was clear from the beginning that to deal with the Tarski's conjectures one
needs to have at least two principal things done:  a precise description of solution sets of systems of
equations over free groups and a robust theory of finitely generated groups which satisfy the same universal (existential) formulas as a free non-abelian group. In the classical  case, algebraic geometry provides a unifying view-point on polynomial equations, while commutative algebra and the elimination theory give the required decision  tools.  Around 1998 three papers appeared almost simultaneously that address analogous issues in the group case.
 Basics of algebraic (or Diophantine) geometry over groups has been outlined  by
Baumslag, Miasnikov and Remeslennikov in \cite{BMR1}, while  the fundamentals of the elimination theory and the theory of fully residually free groups appeared in  the works by Kharlampovich and Miasnikov
\cite{KMNull}, \cite{KMIrc}. These two papers contain results that
became fundamental for the proof of the above two theorems, as well
as in the theory of fully residually free groups.  The goal of these
lectures is to explain why these results are important and to give
some ideas of the proof.

\subsection{Milestones of the theory of equations in free groups}
The first general results on  equations in groups
appeared in the 1960's \cite{Lyndon1}. About this time Lyndon (a former student of Tarski) came up with several extremely
important ideas. One of these is to consider completions of a given group $G$  by extending  exponents into various rings (analogs of extension of ring of scalars in
commutative algebra) and use these completions  to parameterize solutions of equations in $G$. Another idea is to consider groups with free
length functions  with  values in  some ordered abelian group. This allows one to axiomatize the classical Nielsen technique based on the standard length function in free groups and apply it to "non-standard" extensions of free groups, for instance, to ultrapowers of free groups. A link with the Tarski's problems comes here by  the Keisler-Shelah theorem, that states that two groups are elementarily equivalent if and only if their ultrapowers (with respect to a non-principal ultrafilter) are isomorphic.  The
idea to study freely discriminated (fully residually free) groups in
connection to equations in a free group also belongs to Lyndon. He proved \cite{Lyndon2} that the completion $F^{\mathbb{Z}[t]}$ of a free group $F$ by the polynomial ring $\mathbb{Z}[t]$ (now it is called the Lyndon's completion of $F$) is discriminated  by $F$. At the time the Tarski's problems withstood the attack, but these ideas gave birth to several influential theories in modern algebra, which were instrumental in the recent solution of the problems. One of the main ingredients that was lacking at the time was a robust mechanism to solve equations in free groups and a suitable description of the solution sets of equations. The main technical goal of these lectures is to describe a host of methods that altogether give  this mechanism, that we refer to as Elimination Processes.

Also in 1960's Malcev \cite{Malcev} described solutions of the
equation \newline $zxyx^{-1}y^{-1}z^{-1}  = aba^{-1}b^{-1}$ in a
free group, which is the simplest non-trivial quadratic equation in
groups. The description uses the group of automorphisms of the
coordinate group of the equation and the minimal solutions relative
to these automorphisms - a very powerful idea, that nowadays is
inseparable from  the modern approach to equations.
  The first break-through  on Tarski's problem came from Merzljakov (who was a part of Malcev's school  in Novosibirsk). He  proved \cite{Merz} a remarkable theorem that any two nonabelian
free groups of finite rank have the same positive theory and showed that positive formulas in free groups have definable Skolem functions, thus giving quantifier elimination of
positive formulas  in free groups to existential formulas. Recall that the
positive theory of a group
 consists of all positive  (without negations in their normal forms)   sentences
 that are true in this group. These results were
precursors of the current approach to the elementary theory of  a free group.

 In the eighties new crucial concepts were introduced.
  Makanin  proved  \cite{Mak82} the algorithmic decidability of the
Diophantine problem over free groups, and  showed that both, the
universal theory  and the positive theory  of a free group are
algorithmically decidable. He created an extremely
powerful technique (the Makanin elimination process) to deal with equations over free groups.

Shortly afterwards,  Razborov (at the time a  PhD student of
Steklov's Institute, where Makanin held a position)
  described the solution set of an arbitrary system of equations over a free group in terms of what is known now as  Makanin-Razborov diagrams \cite{Razborov1}, \cite{Razborov3}.

 A few years later Edmunds and  Commerford  \cite{ComEd} and
Grigorchuck and  Kurchanov \cite{Gr} described solution sets of arbitrary
quadratic equations over free groups.  These equations came to group theory from topology and their role in  group theory was not altogether  clear then. Now they form one of  the corner-stones of the theory of equations in groups due to their relations to JSJ-decompositions of groups.

\subsection{New age}

These are milestones of the theory of equations in free groups up to 1998. The last missing  principal component in the theory of equations in groups was a general geometric point of view similar to the classical affine algebraic geometry. Back to 1970's Lyndon (again!) was musing on this subject \cite{Lyndon3} but for no avail. Finally, in the late 1990's Baumslag, Kharlampovich, Myasnikov, and Remeslennikov developed the basics  of the algebraic geometry over groups \cite{BMR1,KMNull,KMIrc,KM9,MR2}, introducing analogs of the
standard algebraic geometry notions such as
algebraic sets, the Zariski topology, Noetherian domains,
irreducible varieties, radicals and coordinate groups, rational equivalence, etc.

With all this machinery in place it became possible to make the next
crucial step and tie the algebraic geometry over groups,
Makanin-Razborov process for solving equations, and Lyndon's free
$\mathbb{Z}[t]$-exponential group $F^{\mathbb{Z}[t]}$ into one
closely related theory. The corner stone of this theory is
Decomposition Theorem from \cite{KMIrc} (see Section 4.2  below)
which describes the solution sets  of systems of equations in free
groups in terms of {\em non-degenerate triangular quasi-quadratic}
(NTQ) systems. The coordinate groups of the NTQ systems (later
became also known as {\em residually free towers}) play a
fundamental role in the theory of fully residually free groups, as
well as in the elementary theory of free groups.  The Decomposition
Theorem allows one to look at the processes of the
Makanin-Razborov's type as non-commutative analogs of the classical
elimination processes from algebraic geometry. With this in mind we
refer to   such processes in a
 ll their variations as {\em Elimination Processes} (EP).

In the rest of the notes we discuss more developments of the theory,
focusing mostly on the elimination processes, fully residually free
(limit) groups, and  new techniques that appear here.

\section{Basic notions of algebraic geometry over groups}
\label{se:2-4}

Following \cite{BMR1} and \cite{KM9}  we  introduce here some basic notions of algebraic geometry over
groups.

Let $G$ be a group generated by a finite set $A$,
$F(X)$ be a free group with basis $X = \{x_1, x_2, \ldots  x_n\}$,
we defined $G[X] = G \ast F(X)$ to be a free product of $G$ and
$F(X)$. If $S \subset G[X]$ then the expression $S = 1$ is called
{\em a system of equations} over $G$. As an element of the free
product, the left side of every equation in $S = 1$ can be written
as a product of some elements from $X \cup X^{-1}$ (which are called
{\em variables}) and some elements from $A$ ({\em constants}). To
emphasize this we sometimes write $S(X,A) = 1$.

A {\em solution} of the system $ S(X) = 1$ over a group $G$ is a
tuple of elements $g_1, \ldots, g_n \in G$ such that after
replacement of each $x_i$ by $g_i$ the left hand side of every
equation in $S = 1$ turns into the trivial element of $G$. To study
equations over a given fixed group
 $G$ it is convenient to consider the category of $G$-groups, i.e., groups which
contain the group $G$ as a distinguished subgroup. If $H$ and $K$
are $G$-groups then a homomorphism $\phi: H \rightarrow K$ is a $G$-
homomorphism if $g^\phi = g$ for every $g \in G$, in this event we
write $\phi: H \rightarrow_G K$. In this category morphisms are
$G$-homomorphisms; subgroups are $G$-subgroups, etc. A solution of
the system $S = 1$ over $G$ can be described as a $G$-homomorphism
$\phi : G[X] \longrightarrow G$ such that $\phi(S) = 1$. Denote by
$ncl(S)$ the normal closure of $S$ in $G[X]$, and by $G_S$ the
quotient group $G[X]/ncl(S)$. Then every solution of $S(X) = 1$ in
$G$ gives rise to a $G$-homomorphism $G_S \rightarrow G$, and vice
versa. By $V_G(S)$ we denote the set of all solutions in $G$ of the
system $ S = 1$, it is called the {\em algebraic set defined by}
$S$. This algebraic set $V_G(S)$ uniquely corresponds to the normal
subgroup
$$ R(S) = \{ T(x) \in G[X] \ \mid \ \forall A\in G^n (S(A) = 1
\rightarrow T(A) = 1 \} $$ of the group $G[X]$. Notice that if
$V_G(S) = \emptyset$, then $R(S) = G[X]$. The subgroup $R(S)$
contains $S$, and it is called the {\it radical of $S$}. The
quotient group
$$G_{R(S)}=G[X]/R(S)$$ is the {\em coordinate group} of the
algebraic set  $V(S).$ Again, every solution of $S(X) = 1$ in $G$
can be described as a $G$-homomorphism $G_{R(S)} \rightarrow G$.

 By $Hom_G(H,K)$
we denote the set of all $G$- homomorphisms from $H$ into $K$. It is
not hard to see that the free product $G \ast F(X)$ is a free object
in the category of $G$-groups. This group is called  a free $G$-
group with basis $X$,  and we denote it  by $G[X]$.  A $G$-group $H$
is termed {\em finitely generated $G$-group} if there exists a
finite subset $A \subset H$ such that the set $G \cup A$ generates
$H$. We refer to \cite{BMR1} for a general discussion on $G$-groups.

 A group $G$ is called a \emph{ CSA group} if every maximal abelian
subgroup $M$ of $G$ is malnormal, i.e., $M^g \cap M = 1$ for any $g
\in G - M.$ The abbreviation CSA means conjugacy separability for
maximal abelian subgroups. The class of CSA-groups is quite
substantial. It includes all abelian groups, all torsion-free
hyperbolic groups, all groups acting freely on $\Lambda$-trees  and
many one-relator groups (see, for example, \cite{GKM}.

We define a \emph{Zariski topology} on $G^n$ by taking algebraic
sets in $G^n$ as a sub-basis for the closed sets of this topology.
Namely, the set of all closed sets in the Zariski topology on $G^n$
can be obtained from the set of algebraic sets in two steps:

1) take all finite unions of algebraic sets;

2) take all possible intersections of the sets obtained in step 1).

If $G$ is a non-abelian CSA group and we in the category of
$G$-groups, then the union of two algebraic sets is again algebraic.
Indeed, if $\{w_{i}=1, i\in I\}$ and $\{u_j=1, j\in J\}$ are systems
of equations, then in a CSA group their disjunction is equivalent to
a system
$$[w_i,u_j]=[w_i,u_j^a]=[w_i,u_j^b]=1, \ i\in I,\ j\in J$$ for any two non-commuting
elements $a,b$ from $G$. Therefore the closed sets in the Zariski
topology on $G^n$ are precisely the algebraic sets.

A group $G$ is called \emph {equationally Noetherian} if every
system $S(X) = 1$ with coefficients from $G$ is equivalent over $G$
to a finite subsystem $S_0 = 1$, where $S_0 \subset S$, i.e.,
$V_G(S) = V_G(S_0)$.  It is known that linear groups (in particular,
freely discriminated groups) are equationally Noetherian (see
\cite{Gub}, \cite{Br}, \cite{BMR1}). If $G$ is equationally
Noetherian then the Zariski topology on $G^n$ is {\em Noetherian}
for every $n$, i.e., every proper descending chain of closed sets in
$G^n$ is finite. This implies that every algebraic set $V$ in $G^n$
is a finite union of irreducible subsets (they are called {\it
irreducible components} of $V$), and such decomposition of $V$ is
unique. Recall that a closed subset $V$ is {\it irreducible} if it
is not a union of two proper closed (in the induced topology)
subsets.

\section{Fully residually free groups}
\subsection{Definitions and elementary properties}

 Finitely generated fully residually free groups (limit
groups) play a crucial role in the theory of equations and
first-order formulas over a free group.  It is remarkable that these
groups, which have been widely studied before,  turn out to be  the
basic objects in
 newly developing areas of algebraic geometry and model theory
 of free groups. Recall that a group $G$ is called {\it fully residually
free} (or {\it freely discriminated}, or {\it $\omega$-residually
free}) if for any finitely many non-trivial elements $g_1, \ldots,
g_n \in G$ there exists a homomorphism $\phi$ of $G$ into a free
group $F$, such that $\phi(g_i) \neq 1$ for $i = 1, \ldots, n$. The
next proposition summarizes some simple properties of fully
residually free groups.

\begin{proposition}
  Let $G$ be a fully residually free group. Then $G$ possesses the
following properties.
\begin{enumerate}
  \item \label{e:pr0} $G$ is torsion-free;
  \item \label{e:pr1} Each subgroup of $G$ is a
fully residually free group;
    \item \label{e:pr3} $G$ has the CSA property;
\item \label{e:pr2} Each Abelian subgroup of $G$ is contained in a unique
  maximal finitely generated Abelian subgroup, in particular, each
  Abelian subgroup of $G$ is finitely generated;
  \item \label{e:pr4} $G$ is finitely presented, and has only finitely many conjugacy classes
  of its maximal Abelian subgroups.
  \item \label{e:pr5} $G$ has solvable word problem;
  \item  \label{e:pr6}$G$ is linear;
\item  \label{e:pr7} Every 2-generated subgroup of $G$ is either free or abelian;
\item  \label{e:pr8} If rank ($G$)=3 then either $G$ is free of rank
3, free abelian of rank 3, or a free rank one extension of
centralizer of a free group of rank 2 (that is $G=\langle
x,y,t|[u(x,y),t]=1\rangle$ , where the word $u$ is not a proper
power).
\end{enumerate}
\end{proposition}

 Properties~\ref{e:pr0} and~\ref{e:pr1} follow immediately from the
definition of an $\mathcal{F}$-group. A proof of
property~\ref{e:pr3} can be found in~\cite{BMR1};
property~\ref{e:pr2} is proven in~\cite{KMIrc}.
Properties~\ref{e:pr2} and ~\ref{e:pr4} are proved in~\cite{KMIrc}.
Solvability of the word problem follows from ~\cite{Mak84} or from
residual finiteness of a free group. Property ~\ref{e:pr8} is proved
in \cite{FGMRS}. Property \ref{e:pr6} follows from linearity of $F$
and property 6 in the next proposition. The ultraproduct of
$SL_2({\mathbb Z})$ is $SL_2(^*{\mathbb Z})$, where $^*{\mathbb Z}$
is the ultpaproduct of ${\mathbb Z}.$ (Indeed, the direct product
$\prod SL_2({\mathbb Z})$ is isomorphic to $SL_2({\prod\mathbb Z}).$
Therefore, one can define a homomorphism from the ultraproduct of
$SL_2({\mathbb Z})$ onto $SL_2(^*{\mathbb Z})$. Since the
intersection of a finite number of sets from an ultrafilter again
belongs to the ultrafilter, this epimorphism is a monomorphism.)
Being finitely generated $G$ embeds in $SL_2(R)$, where $R$ is a
finitely generated subring in $^*{\mathbb Z}$.

\vspace{2mm} \begin{proposition}(no coefficients)  Let $G$ be a
finitely generated group. Then the following conditions are
equivalent:
\begin{itemize}
 \item [1)]  $G$ is freely  discriminated (that is for finitely many non-trivial elements $g_1,\ldots ,g_n\in G$ there exists a homomorphism $\phi$ from
 $G$ to a free group such that $\phi (g_i)\neq 1$ for $i=1,\ldots ,n$);
 \item [2)] [Remeslennikov] $G$ is universally equivalent to $F$ (in the language without constants);
 \item [3)] [Baumslag, Kharlampovich, Myasnikov, Remeslennikov] $G$ is the coordinate group of an irreducible variety over
 a free group.
 \item [4)]  [Sela] $G$ is a limit group (to be defined in the proof of proposition 3).
 \item [5)] [Champetier and Guirardel] $G$ is a limit of free groups in Gromov-Hausdorff
 metric (to be defined in the proof of proposition 3).
 \item [6)] $G$ embeds into an ultrapower of free
 groups.
 \end{itemize}
\end{proposition}
\vspace{2mm} \begin{proposition} (with coefficients) Let $G$ be a
finitely generated group containing a free non-abelian group $F$ as
a subgroup. Then the following conditions are equivalent:

\begin{itemize}
 \item [1)]  $G$ is $F$-discriminated by $F$;
 \item [2)] [Remeslennikov] $G$ is universally equivalent to $F$ (in the language with constants);
 \item [3)] [Baumslag, Kharlampovich, Myasnikov, Remeslennikov] $G$ is the coordinate group of an irreducible variety over
 $F$.
 \item [4)]  [Sela] $G$ is a restricted limit group.
 \item [5)] [Champetier and Guirardel] $G$ is a limit of free groups in Gromov-Hausdorff
 metric.
 \item [6)] $G$ $F$-embeds into an ultrapower of $F$.
\end{itemize}
\end{proposition}
We will prove Proposition 3, the proof of Proposition 2 is very
similar. We will first prove the equivalence 1)$\Leftrightarrow$ 2).
Let $L_A$ be the language of group theory with generators $A$ of $F$
as constants.
 Let $G$ be a f.g. group which is $F$-discriminated by $F$.
  Consider a formula
\[
\exists X (U(X,A)=1 \wedge W(X,A)\ne 1).
\]
If this formula is true in $F$, then it is also true in $G$, because
$F\leq G$. If it is true in $G$, then for some $\bar X\in G^m$ holds
$U(\bar X,A)=1$ and $W(\bar X,A)\neq 1$. Since $G$ is
$F$-discriminated by $F$, there is an $F$-homomorphism $\phi\colon
G\to F$ such that $\phi (W({\bar X},A))\neq 1$, i.e. $W({\bar
X}^{\phi},A)\neq 1$.  Of course $U({\bar X}^{\phi},A)= 1.$ Therefore
the above formula is true in $F$. Since in $F$-group a conjunction
of equations [inequalities] is equivalent to one equation [resp.,
inequality], the same existential sentences in the language $L_A$
are true in $G$ and in $F$.

Suppose now that $G$ is $F$-universally equivalent to $F$. Let
$G=\langle X, A\mid S(X,A)=1\rangle $, be a presentation of $G$ and
$w_1(X,A),\dots ,w_k(X,A)$ nontrivial elements in $G$. Let $Y$ be
the set of the same cardinality as $X$. Consider a system of
equations $S(Y,A)=1$ in variables $Y$ in $F$. Since the group $F$ is
equationally Noetherian, this system is equivalent over $F$ to a
finite subsystem $S_1(Y,A)=1$. The formula
\[
\Psi =\forall Y(S_1(Y,A)=1\to \bigl(w_1(Y,A)=1\vee\dots \vee
w_k(Y,A)=1)\bigr).
\]
is false in $G$, therefore it is false in $F$. This means that there
exists a set of elements $B$ in $F$ such that $S_1(B,A)=1$ and,
therefore, $S(B,A)=1$ such that $w_1(B,A)\ne 1\wedge\dots\wedge
w_k(B,A)\ne 1$. The map $X\to B$ that is identical on $F$ can be
extended to the $F$-homomorphism from $G$ to $F$.

 1)$\Leftrightarrow$ 3) Let $H$ be an equationally Noetherian CSA-group. We will prove that  $V(S)$ is
  irreducible if and only if $H_{R(S)}$ is discriminated in $H$ by
  $H$-homomorphisms.

  Suppose $V(S)$ is not irreducible and $V(S)=V(S_1)\cup V(S_2)$
  is its decomposition into proper subvarieties.  Then there exist $s_i\in R
  (S_i)\setminus R(S_j), j\ne i$. The set $\{s_1, s_2\}$ cannot be discriminated in $H$ by $H$-homomorphisms.

  Suppose now $s_1,\dots ,s_n$ are elements such that for any retract
  $f\colon H_{R(S)}\to H$ there exists $i$ such that $f(s_i)=1$; then
  $V(S)=\bigcup_{i=1}^mV(S\cup s_i).$ $\ \ \ \Box$

 Sela \cite{Sela1} defined limit groups as follows.
Let $G$ be a f.g. group and let $\{\phi_j\}$ be a sequence of
homomorphisms from $G$ to a free group $F$ belonging to distinct
conjugacy classes (distinct $F$-homomorphisms belong to distinct
conjugacy classes).

$F$ acts by isometries on its Cayley graph $X$ which is a simplicial
tree. Hence, there is a sequence of actions of $G$ on $X$
corresponding to $\{\phi_j\}$.

By rescaling metric on $X$ for each $\phi_j$ one obtains a sequence
of simplicial trees $\{X_j\}$ and a corresponding sequence of
actions of $G$. $\{X_j\}$ converges to a real tree $Y$
(Gromov-Hausdorff limit) endowed with an isometric action of $G$.
The kernel of the action of $G$ on $Y$ is defined as
$$K = \{g \in G \mid g y = y,\ \forall y \in Y\}.$$

Finally, $G / K$ is said to be the {\bf limit} group (corresponding
to $\{\phi_j\}$ and rescaling constants). We will prove now the
equivalence 1)$\Leftrightarrow$ 4). A slight modification of the
proof below should be made to show that limit groups are exactly
f.g. fully residually free groups.

 Suppose that
$G=\langle g_1,\ldots ,g_k\rangle $ is f.g. and discriminated by
$F$. There exists a sequence of homomorphisms $\phi _n:G\rightarrow
F,$ so that  $\phi _n$ maps the elements in a ball of radius $n$ in
the Cayley graph of $G$ to distinct elements in $F$. By rescaling
the metric on $F$, we obtain a subsequence of homomorphisms $\phi
_m$ which converges to an action of a limit group $L$ on a real tree
$Y$. In general, $L$ is a quotient of $G$, but since the
homomorphisms were chosen so that $\phi _n$ maps a ball of radius
$n$ monomorphically into $F$, $G$ is isomorphic to $L$ and,
therefore, $G$ is a limit group.

To prove the converse, we need the fact  (first proved in
\cite{Sela1}) that a f.g. limit group is finitely presented. We may
assume further that a limit group $G$ is non-abelian because the
statement is, obviously, true for abelian groups. By definition,
there exists a f.g. group $H$, an integer $k$ and a sequence of
homomorphisms $h_k:H\rightarrow F$, so that the limit of the actions
of $H$ on the Cayley graph of $F$ via the homomorphisms $h_k$ is a
faithful action of $G$ on some real tree $Y$. Since $G$ is finitely
presented for all but finitely many $n$, the homomorphism $h_n$
splits through the limit group $G$, i.e. $h_n=\phi\psi _n$, where
$\phi: H\rightarrow G$ is the canonical projection map, and the
$\psi _n$'s are homomorphisms $\psi _n:G\rightarrow F$. If $g\neq 1$
in $G$, then for all but finitely many $\psi _n$'s $g^{\psi _n}\neq
1.$ Hence, for every finite set of elements $g_1,\ldots ,g_m\neq 1$
in $G$ for all but finitely many indices $n$, $g_1^{\psi _n},\ldots
,g_m^{\psi _n}\neq 1,$ so $G$ is $F$ discriminated. $\ \ \Box$

The equivalence 2)$\Leftrightarrow$ 6) is a particular case of
general results in model theory (see for example \cite{BS69} Lemma
3.8 Chap.9). $\ \ \Box$

5)$\Leftrightarrow$ 6).
 Champetier and  Guirardel \cite{CG} used another approach to limit
groups.

A {\em marked} group $(G,S)$ is a group $G$ with a prescribed family
of generators $S = (s_1,\ldots,s_n)$.

Two marked groups $(G, (s_1,\ldots,s_n))$ and $(G',
(s'_1,\ldots,s'_n))$ are isomorphic as marked groups if the
bijection $s_i \longleftrightarrow s'_i$ extends to an isomorphism.
For example, $(\langle a \rangle,(1,a))$ and $(\langle a
\rangle,(a,1))$ are not isomorphic as marked groups. Denote by
${\cal G}_n$ the set of groups marked by $n$ elements up to
isomorphism of marked groups.

One can define a metric on ${\cal G}_n$ by setting the distance
between two marked groups $(G, S)$ and $(G',S')$ to be $e^{-N}$ if
they have exactly the same relations of length at most $N$ (under
the bijection $S \longleftrightarrow S'$).

Finally, a limit group in their terminology is a limit (with respect
to the metric above) of marked free groups in ${\cal G}_n$.

It is shown in \cite{CG} that a group is a limit group if and only
if it is a finitely generated subgroup of an ultraproduct of free
groups (for a non-principal ultrafilter), and any such ultraproduct
of free groups contains all the limit groups. This implies the
equivalence 5)$\Leftrightarrow$ 6). $\ \ \Box$

Notice that ultrapowers of a free group have the same elementary
theory as a free group by Los' theorem.

First non-free finitely generated examples of fully residually free
groups, that include all non-exceptional surface groups, appeared in
\cite{Bau62}, \cite{Bau67}. They obtained fully residually free
groups as subgroups of free extensions of centralizers in free
groups.

\subsection{Lyndon's completion $F^{{\mathbb Z}[t]}$}

Studying equations in free groups Lyndon \cite{L} introduced the
notion of a group with parametric exponents in  an associative
unitary ring $R$. It can be defined as a union of the chain of
groups
$$F = F_0 < F_1 < \cdots < F_n < \cdots,$$
where $F = F(X)$ is a free group on an alphabet $X$, and $F_k$ is
generated by $F_{k-1}$ and formal expressions of the type
$$\{ w^\alpha \mid w \in F_{k-1},\ \alpha \in R\}.$$
That is, every element of $F_k$ can be viewed as a {parametric word}
of the type
$$w_1^{\alpha_1} w_2^{\alpha_2} \cdots w_m^{\alpha_m},$$
where $m \in \mathbb{N},\ w_i \in F_{k-1}$, and $\alpha_i \in R$. In
particular, he described free exponential groups $F^{\mathbb{Z}[t]}$
over the ring of integer polynomials $\mathbb{Z}[t]$.   Notice that
ultrapowers of free groups are operator groups over ultraproducts of
$\mathbb{Z}$.

 In the same paper Lyndon proved an amazing result  that $F^{\mathbb{Z}[t]}$ is fully residually free.
 Hence all  subgroups of $F^{\mathbb{Z}[t]}$ are fully residually free.
Lyndon showed that solution sets of one variable equations can be
described in terms of parametric words. Later it was shown in
\cite{Appel} that coordinate groups of irreducible one-variable
equations are just extensions of centralizers in $F$ of rank one
(see the definition in the second paragraph below). In fact, this
result is not entirely accidental, extensions of centralizers play
an important part here. Recall that Baumslag \cite{Bau62} already
used them in proving that surface groups are freely discriminated.

Now, breaking the  natural march of history, we go ahead of time and formulate one crucial result which justifies our  discussion on Lyndon's completion $F^{\mathbb{Z}[t]}$ and highlights the role of the group $F^{\mathbb{Z}[t]}$ in the whole subject.

{\bf Theorem} (The Embedding Theorem \cite{KMIrc},\cite{KM9}) {\em
Given an irreducible system $S = 1$
 over $F$  one can effectively embed the coordinate group $F_{R(S)}$  into
$ F^{Z[t]}$.}

 A modern treatment of exponential groups
was done by Myasnikov and Remeslennikov \cite{MR1} who proved that
the group $F^{\mathbb{Z}[t]}$ can be obtained from $F$ by an
infinite chain of HNN-extensions of a very specific type, so-called
{\it extensions of centralizers}:
 $$F = G_0 <  G_1 < \ldots <  \ldots \cup G_i = F^{Z[t]}$$
  where
   $$G_{i+1} = \langle G_i, t_i \mid [C_{G_i}(u_i),t_i] =  1\rangle.$$
(extension of the centralizer $C_{G_i}(u_i)$, where $u_i\in G_i$).

This implies that finitely generated subgroups of
$F^{\mathbb{Z}[t]}$ are, in fact,
 subgroups of $G_i$. Since $G_i$ in an HNN-extension, one can apply Bass-Serre theory to
 describe the structure of these subgroups. In fact, f.g. subgroups of $G_i$ are fundamental groups of graphs of groups induced by the
 HNN structure of $G_i$. For instance, it is routine
 now to show that  all such subgroups $H$ of $G_i$ are finitely presented. Indeed, we only have to show that the intersections $G_{i-1}\cap H^g$ are finitely
 generated. Notice, that if in the amalgamated product amalgamated subgroups are finitely generated and one of the factors is not, then the alamgamated product is not finitely
 generated (this follows from normal forms of elements in the amalgamated products). Similarly, the base group of a f.g. HNN extension with f.g. associated subgroups
 must be f.g. Earlier Pfander \cite{Pf} proved that f.g. subgroups of the free ${\mathbb Z}[t]$-group on two generators
 are finitely presented.
  Description of f.g. subgroups of $F^{\mathbb{Z}[t]}$ as fundamental groups of graphs of groups implies immediately that such groups  have non-trivial
   abelian splittings (as amalgamated product or HNN extension with abelian edge group which is maximal abelian in one of the base subgroups).
    Furthermore, these groups can be obtained from free groups by finitely many free constructions (see next section).

 The original Lyndon's result on fully residual freeness of $F^{\mathbb{Z}[t]}$  gives decidability of the Word Problem in $F^{\mathbb{Z}[t]}$,
 as well as in all its subgroups.  Since any fully residually free group given by a finite presentation with relations $S$ can be presented as the coordinate group $F_{R(S)}$
 of a coefficient-free system $S=1$. The Embedding Theorem then implies decidability of WP in arbitrary f.g. fully residually residually free group.

   The Conjugacy Problem is also decidable in $F^{\mathbb{Z}[t]}$ - but this was proved much later, by Ribes and Zalesski in \cite{RZ}.
   A similar, but stronger,  result is due to Lyutikova who
   showed in  \cite{Lut} that the Conjugacy Problem  in $F^{\mathbb{Z}[t]}$ is residually free, i.e.,  if two elements $g, h$ are not conjugate in
   $F^{\mathbb{Z}[t]}$ (or in $G_i$) then there is an $F$-epimorphism $\phi: F^{\mathbb{Z}[t]} \to F$  such that $\phi(g)$ and $\phi(h)$ are not
   conjugate in $F$. Unfortunately, this does not imply immediately that the CP in subgroups of $F^{\mathbb{Z}[t]}$ is  also residually free,
   since two elements may be not conjugated in a subgroup $H \leq F^{\mathbb{Z}[t]}$, but conjugated in the whole group $F^{\mathbb{Z}[t]}$.
   We discuss CP in arbitrary f.g. fully res. free groups in Section 5.

\section{Main results in \cite{KMIrc}}

\subsection{Structure and embeddings}

In 1996 we proved  the
converse of the Lyndon's result mentioned above, every finitely
generated fully residually free group is embeddable into
$F^{\mathbb{Z}[t]}$.

\begin{theorem}  \cite{KMIrc},\cite{KM9} Given an irreducible system
$S = 1$
 over $F$  one can effectively embed the coordinate group $F_{R(S)}$  into
$ F^{Z[t]}$ i.e.,   one can find $n$ and an embedding $F_{R(S)}
\rightarrow G_n$ into an iterated centralizer extension $G_n$.
\end{theorem}

\begin{cor} For every freely indecomposable non-abelian finitely generated fully
residually free group one can effectively find a non-trivial
splitting (as an amalgamated product or HNN extension) over a cyclic
subgroup.\end{cor}

\begin{cor}  Every finitely generated fully residually free
group is finitely presented. There is an algorithm  that, given a
presentation of a f.g. fully residually free group $G$ and
generators of the subgroup $H$, finds a finite presentation for $H$.
\end{cor}
\begin{cor} Every finitely generated residually free group $G$ is a subgroup of a direct product of finitely
many fully residually free groups; hence, $G$ is embeddable into $
F^{Z[t]}\times\ldots\times F^{Z[t]}$. If $G$ is given as a
coordinate group of a finite system of equations, then this
embedding can be found effectively.
\end{cor}

Indeed, there exists a finite system of coefficient free equations
$S=1$ such that $G$ is a coordinate group of this system, and
$ncl(S)=R(S)$. If $V(S)=\cup _{i=1}^nV(S_i)$ is a representation of
$V(S)$ as a union of irreducible components, then $R(S)=\cap
_{i=1}^n R(S_i)$ and $G$ embeds into a direct product of coordinate
groups of systems $S_i=1$, $i=1,\ldots ,n.$

 This allows one to study the coordinate
groups of irreducible systems of equations (fully residually free
groups) via their splittings into graphs of groups. This also
provides a complete description of finitely generated
 fully residually free
groups and gives a lot of information about their algebraic
structure.  In particular, they act freely on $\mathbb{Z}^n$-trees,
and all these groups, except for abelian and surface groups, have a
non-trivial cyclic JSJ-decomposition.

Let $K$ be an HNN-extension of a group $G$ with associated subgroups
$A$ and $B$. $K$ is called a separated HNN-extension if for any
$g\in G$, $A^g\cap B=1$.

\begin{cor} Let a group $G$ be obtained from a free group $F$
by finitely many centralizer extensions. Then every f. g. subgroup
$H$ of $G$ can be obtained from free abelian groups of finite rank
by finitely many operations of the following type: free products,
free products with abelian amalgamated subgroups at least one of
which is a maximal abelian subgroup in its factor, free extensions
of centralizers, separated HNN-extensions with abelian associated
subgroups at least one of which is maximal. \end{cor}

\begin{cor}(Groves, Wilton \cite{GW}) One can enumerate all finite presentations of fully
residually free groups.
\end{cor}

Theorem 3 is proved as a corollary of Theorem 6 below.

\begin{cor} Every f.g. fully residually free group acts freely on some  $\mathbb{Z}^n$-tree with lexicographic order for a suitable $n$.\end{cor}
Hence, a simple application of the change of the group functor shows
that $H$ also acts  freely on an
 $\mathbb{R}^n$-tree.   Recently, Guirardel proved the latter  result independently using different techniques \cite{Gui}.
 It is worthwhile  to mention here that free group  actions on $\mathbb{Z}^n$-trees give a tremendous amount of information on the group
 and its subgroups, especially with regard to various algorithmic problems (see Section 5).

    Notice, that there are f.g. groups acting freely on $\mathbb{Z}^n$-trees which are not fully residually free (see conjecture (2) from Sela's list
    of open problems). The simplest example is the group of closed non-orientable surface of genus 3.
    In fact, the results in \cite{KMSa, KMSb} show that there are very many groups like that -  the class of groups acting freely on
    $\mathbb{Z}^n$-trees is much wider than the class of fully residually free groups.
    This class deserves a separate discussion, for which we refer to \cite{KMSa, KMSb}.
 Combining Corollary 4 with the results from \cite{KMComb} or \cite{BFComb} we proved in \cite{KMIrc} that f.g. fully residually free groups
 without subgroups $\mathbb{Z} \times \mathbb{Z}$ (or equivalently, with cyclic maximal abelian subgroups)  are hyperbolic.
 We will see in Section 4.2 that this has some implication on the structure of the models of the  $\forall \exists$-theory of a given
 non-abelian free group.  Later,  Dahmani \cite{Dah} proved a generalization of this, namely, that an arbitrary f.g. fully residually free group is  hyperbolic relative to its  maximal abelian non-cyclic subgroups.

 Recently  N. Touikan described coordinate groups of  two-variable equations \cite{Toui}.

\subsection{Triangular quasi-quadratic systems}

We use an Elimination Process to transform systems of equations.
Elimination Process EP is a symbolic rewriting process of a certain
type that transforms formal systems of equations in groups or
semigroups. Makanin (1982) introduced the initial version of the EP.
This  gives a decision algorithm to verify consistency of a given
system - {\em
 decidability of the Diophantine problem} over free groups.
He estimates  the length of the minimal solution (if it exists).
Makanin introduced the fundamental  notions: generalized equations,
elementary and entire transformations, notion of complexity.
 Razborov (1987) developed EP much further.
 Razborov's EP produces {\em
 all  solutions} of a given system in $F$.
He used special groups of automorphisms, and fundamental sequences
to encode solutions.

We obtained in 1996 \cite{KMIrc} an effective description of
solutions of equations in free (and fully residually free ) groups
in terms of very particular {\em triangular systems} of equations.
First, we give a definition.

{\bf  Triangular quasi-quadratic (TQ)} system is a finite system
that has the following form

\medskip

$S_1(X_1, X_2, \ldots, X_n,A) = 1,$

\medskip
$\ \ \ \ \ S_2(X_2, \ldots, X_n,A) = 1,$

$\ \ \ \ \ \ \ \ \ \  \ldots$

\medskip
$\ \ \ \ \ \ \ \ \ \ \ \ \ \ \ \ S_n(X_n,A) = 1$

\medskip \noindent
 where either $S_i=1$ is quadratic in variables $X_i$, or $S_i=1$ is a system of commutativity equations for all variables from $X_i$
 and, in addition, equations $[x,u]=1$ for all $x\in X_i$ and some
 $u\in F_{R(S_{i+1},\ldots ,S_n)}$ or $S_i$ is empty.

 A TQ system above is { non-degenerate ( NTQ)} if for every $i$,
   $S_i(X_i, \ldots, X_n,A) = 1$ has a
  solution in the coordinate group $F_{R(S_{i+1}, \ldots, S_n)}$,
  i.e.,
$S_i = 1$ (in algebraic geometry one would say that a solution
exists in a generic point of the system $S_{i+1} = 1, \ldots,
  S_n = 1$).

We proved in \cite{KMNull} (see also \cite{KM9}) that {\em NTQ
systems are irreducible} and, therefore, their coordinate groups
(NTQ groups) are fully residually free. (Later Sela called NTQ
groups {\em $\omega$-residually free towers} \cite{Sela1}.)

We represented a solution set of a system of equations canonically
as a union of solutions of
 a finite family of NTQ groups.

\begin{theorem}\cite{KMIrc}, \cite{KM9} One can effectively construct EP that starts on an arbitrary system
     $$S(X,A)  = 1$$
      and results in finitely many   NTQ systems
    $$U_1(Y_1) =1, \ldots, U_m(Y_m) = 1$$
     such that
   $$V_F(S) = P_1(V(U_1)) \cup \ldots \cup P_m((U_m))$$
     for some word mappings $P_1, \ldots, P_m$. ($P_i$ maps a tuple $\bar Y_i\in V(U_i)$ to a tuple $\bar X\in V(S)$. One can think about
     $P_i$ as an $A$-homomorphism from $F_{R(S)}$into $F_{R(U_i)}$, then any solution $\psi :F_{R(U_i)}\rightarrow F$ pre-composed with $P_i$ gives a
     solution $\phi: F_{R(S)}\rightarrow F$. )
\end{theorem}

    Our elimination process can be viewed as a
{ non-commutative  analog  of the  classical elimination process in
algebraic geometry}.

  Hence, going "from
 the  bottom to the top" every solution
  of the subsystem $S_n = 1, \ldots S_i = 1$ can be extended to a solution of the next
  equation $S_{i-1} = 1$.

\begin{theorem}\cite{KMIrc}, \cite{KM9}
  All solutions of the system of equations $S=1$ in $F(A)$
 can be effectively represented as homomorphisms from $F_{R(S)}$ into $F(A)$ encoded
 into the following finite canonical Hom-diagram. Here all groups, except, maybe, the one in the root, are fully residually free,
 (given by a finite presentation) arrows pointing down
 correspond to epimorphisms (defined effectively in terms of generators) with non-trivial kernels, and loops
 correspond to automorphisms of the coordinate groups.

\let\Om\Omega
\let\si\sigma

$$
\xymatrix@C=0.2cm{
& & F_{R(S)} \ar[dl] \ar[dr] \ar[drrrrr] \\
& F_{R(\Om_{v_1})} \ar@(ur,ul)
 []_{\si_1} \ar[dl] \ar[dr]
&& F_{R(\Om_{v_2})} && \cdots && F_{R(\Om_{v_n})} \\
F_{R(\Om_{v_{21}})} &
\cdots & F_{R(\Om_{v_{2m}})} \ar[dl] \ar[dr] \ar@(ur,ul)[]_{\si_2} \\
&& \cdots & \\
&\cdots& F_{R(\Om_{v_k})} \ar[d] \\
&& F(A) * F(T) \ar[d] \\
&& F(A) }
$$
\end{theorem}

A family of homomorphisms encoded in a path from the root to a leaf
of this tree is called a {\em fundamental sequence} or {\em
fundamental set}  of solutions (because each homomorphism in the
family is a composition of a sequence of automorphisms and
epimorphisms). Later Sela called such family a {\em resolution}.
Therefore the solution set of the system $S=1$ consists of a finite
number of fundamental sets. And each fundamental set "factors
through"  one of the NTQ systems from Theorem 4. If $S=1$ is
irreducible, or, equivalently, $G=F_{R(S)}$ is fully residually
free, then, obviously, one of the fundamental sets discriminates
$G$. This gives the following result.

\begin{theorem} \cite{KMIrc}, \cite{KM9}   Finitely generated fully residually free groups are subgroups of
coordinate groups of NTQ systems. There is an algorithm to construct
an embedding.
\end{theorem}
  This corresponds  to the {extension
  theorems} in the classical theory of elimination for polynomials.
  In \cite{KMNull} we have shown that an NTQ group can be embedded
  into a group obtained from a free group by a series of extensions
  of centralizers. Therefore Theorem 3 follows from Theorem 6.

  Since
  NTQ groups are fully residually free, fundamental sets
  corresponding to different NTQ groups in Theorem 4 discriminate
  fully residually free groups which are coordinate groups of
  irreducible components of system $S(X,A)=1.$ This implies

 \begin{theorem}  \cite{KMIrc}, \cite{KM9}  There is an algorithm to find irreducible components for a system
of equations over a free group.
\end{theorem}

Now we will formulate a technical result which is the keystone in
the proof of Theorems 4 and 5. An elementary abelian splitting of a
group is the splitting as an amalgamated product or HNN-extension
with abelian edge group. Let $G=A*_{C}B$ be an elementary abelian
splitting of $G$.
  For $c\in C$ we  define an automorphism $\phi_c :G\rightarrow G$
 such that $\phi_c(a)=a$ for $a\in A$ and $\phi_c(b)=b^{c}=c^{-1}bc$ for  $b\in B$.

If $G=A*_{C}=\langle A,t|c^{t}=c', c\in C\rangle$ then for $c \in C$
define $\phi_c :G\rightarrow G$ such that  $\phi_c (a)=a$ for $a\in
A$ and  $\phi_c (t)=ct$.

We call $\phi_c$ a {\em Dehn twist} obtained from the corresponding
elementary abelian splitting of $G$. If $G$ is an $F$-group, where
$F$ is a subgroup of one of the factors  $A$ or $B$, then Dehn
twists that fix elements of the free group $F\leq A$ are
 called  {\em canonical Dehn twists}.

If $G=A*_{C}B$ and $B$ is a maximal abelian subgroup of $G$, then
every automorphism of $B$ acting trivially on $C$ can be extended to
the automorphism of $G$ acting trivially on $A$. The subgroup of
$Aut (G)$ generated by such automorphisms and canonical Dehn twists
is called the group of canonical automorphisms of $G$.

 Let $G$ and $K$  be  $H$-groups and  ${\mathcal A} \leq
Aut_H(G)$ a group
 of \newline $H$-automorphisms of $G$. Two
$H$-homomorphisms $\phi$ and $\psi$ from $G$ into $K$ are called
{\em $A$-equivalent} (symbolically, $\phi \sim_{\mathcal A} \psi$)
if there exists $\sigma\in {\mathcal A}$ such that $\phi=
\sigma\psi$ (i.e., $g^\phi
 = g{^{\sigma\psi}}$ for $g \in G$). Obviously, $\sim_{\mathcal A}$ is an
equivalence relation on $Hom_H(G,K)$.

 Let  $G$ be a fully residually $F$ group ($F=F(A)\leq G$) generated by a finite set $X$ (over $F$)
 and  ${\mathcal A}$  the group of canonical $F$ automorphisms of $G$. Let  $ \bar F = F(A\cup Y)$
 a free group with basis $A \cup Y$ (here $Y$ is an arbitrary set) and
$\phi_1, \phi_2 \in Hom_F(G,\bar F)$ .  We write $\phi _1 < \phi _2$
if there exists an automorphism $\sigma\in {\mathcal A}$ and an
$F$-endomorphism $\pi \in Hom_F(\bar F, \bar F)$   such that $\phi
_2=\sigma^{-1} \phi _1\pi$ and
 $$\sum_{x \in X} | x^{\phi _1}| < \sum_{x \in X} | x^{\phi _2}|.$$

\begin{figure}[here]
\centering{\mbox{\psfig{figure=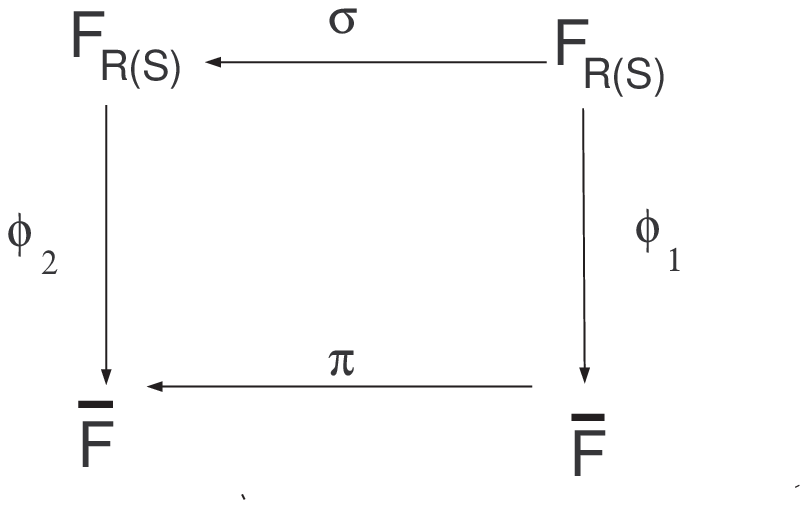,height=2in}}} \caption{$\phi
_1 < \phi _2$} \label{ET0.}
\end{figure}
  An $F$- homomorphism  $\phi :G\rightarrow \bar F$ is
  called {\em minimal} if there is no
  $\phi _1$ such that $\phi _1 < \phi$. In particular,
if  $S(X,A)=1$ is a system of  equations over $F=F(A)$ and
$G=F_{R(S)}$ then $X \cup A$ is a generating set for $G$ over $F$.
In this event, one can consider {\em minimal solutions} of $S = 1$
in $\bar F$.
\begin{definition}
\label{de:max-standard}
 Denote by  $R_{\mathcal A}$   the intersection of
the kernels of all minimal  (with respect to ${\mathcal A}$)
$F$-homomorphisms from $Hom_F(G,\bar F)$. Then $G/R_{\mathcal A}$ is
called the {\em maximal standard quotient} of $G$  and the canonical
epimorphism $\eta:G \rightarrow G/R_{\mathcal A}$ is the {\em
canonical projection}.
\end{definition}
\begin{theorem}\cite{KMIrc} The maximal standard quotient of a
finitely generated fully residually free group is a proper quotient
and can be effectively constructed.
\end{theorem}
This result (without the algorithm) is called the "shortening
argument" in Sela's approach.
\section{Elimination process}

Given a system $S(X) = 1$ of equations in a free group $F(A)$ one
can effectively construct a finite set of  generalized equations
$$\Omega_1, \ldots, \Omega_k$$
 (systems of equations of a particular type) such that:

 \begin{itemize}
    \item given a solution of $S(X) = 1$ in $F(A)$ one can effectively construct a reduced solution of one of $\Omega_i$ in the free semigroup  with basis $A \cup A^{-1}$.
     \item  given a solution of some  $\Omega_i$ in the free semigroup  with basis $A \cup A^{-1}$ one can effectively construct a solution of $S(X) = 1$ in $F(A)$.
\end{itemize}

This is done as follows. First, we replace the system $S(X) = 1$ by
a system of equations, such that each of them has length 3. This can
be easily done by adding new variables. For one equation of length 3
we can construct a generalized equation as in the example below. For
a system of equations we construct it similarly (see \cite{KMIrc}).

{\bf Example.} Suppose we have the simple equation $xyz = 1$ in a
free group. Suppose that we have a solution to this equation denoted
by $x^\phi,y^\phi,z^\phi$ where is $\phi$ is a given homomorphism
into a free group $F(A)$.  Since $x^\phi,y^\phi,z^\phi$ are reduced
words in the generators $A$ there must be complete cancellation. If
we take a concatenation of the geodesic subpaths corresponding to
$x^{\phi}, y^{\phi}$ and $z^{\phi}$ we obtain a path in the Cayley
graph corresponding to this complete cancellation. This is called a
cancellation tree. Then $x^{\phi} = \lambda_1\circ \lambda_2$,
$y^{\phi} = \lambda_2^{-1}\circ \lambda_3$ and $z^{\phi} =
\lambda_3^{-1}\circ \lambda_1^{-1}$, where $u\circ v$ denotes the
product of reduced words $u$ and $v$ such that there is no
cancellation between $u$ and $v$. In the case when all the words
$\lambda_1, \lambda_2, \lambda _3$ are non-empty, the generalized
equation would be the interval in Fig. 2.
\begin{figure}[here]
\centering{\mbox{\psfig{figure=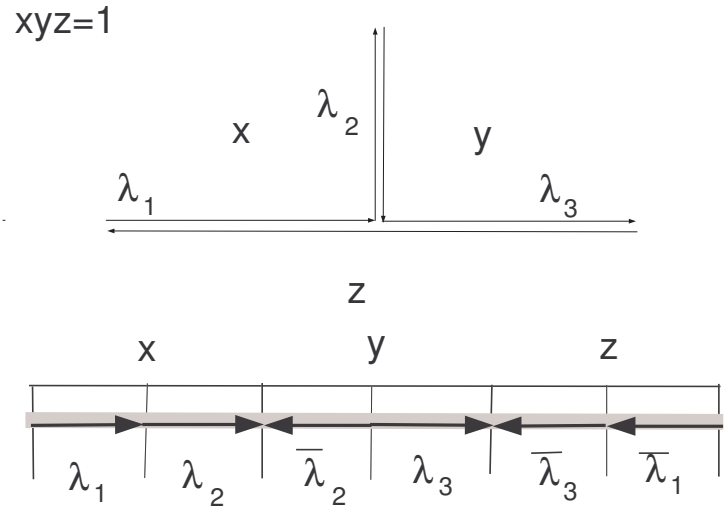,height=3in}}} \caption{From the
cancellation tree for the equation $xyz=1$ to the generalized
equation ($x^{\phi}=\lambda _1\circ\lambda _2,\ y^{\phi}=\lambda
_2^{-1}\circ\lambda _3,\ z^{\phi}=\lambda _3^{-1}\circ\lambda
_1^{-1}).$} \label{ET0}
\end{figure}

Given a generalized equation $\Omega$  one can apply {elementary
transformations} (there are only finitely many of them) and get a
new generalized equation $\Omega^\prime$. If $\sigma$ is a solution
of $\Omega$, then elementary transformation transforms $\sigma$ into
 $\sigma^\prime$.
$$(\Omega, \sigma) \rightarrow (\Omega^\prime, \sigma^\prime).$$
 Elimination Process is a branching process such that on each step one of the finite number of elementary transformations is
applied according to some precise rules to a generalized equation on
this step.

$$\Omega_0 \rightarrow \Omega_1   \rightarrow  \ldots  \rightarrow \Omega_k.$$

From the group theoretic view-point the elimination process  tells
something about the coordinate groups of the systems involved.

This allows one to   transform  the pure combinatorial and
algorithmic results obtained in the elimination process into
statements about the coordinate groups.

\subsection{Generalized equations}
\label{se:4-1}

\begin{definition}
A combinatorial generalized equation $\Omega$ consists of the
following components:
\begin{enumerate}
\item A finite set of {\bf bases} $BS = BS(\Omega)$. The set of
bases ${\mathcal M}$ consists of $2n$ elements ${\mathcal M} =
\{\mu_1, \ldots, \mu_{2n}\}$. The set ${\mathcal M}$ comes equipped
with two functions: a function $\varepsilon: {\mathcal M}
\rightarrow \{1,-1\}$ and an involution $\Delta: {\mathcal M}
\rightarrow {\mathcal M}$ (that is, $\Delta$ is a bijection such
that $\Delta^2$ is an identity on ${\mathcal M}$). Bases $\mu$ and
$\Delta(\mu)$ (or $\bar\mu$) are called {\it dual bases}. We denote
bases by letters $\mu, \lambda$, etc.

\item A set of {\bf boundaries} $BD = BD(\Omega)$. $BD$ is  a finite
initial segment of the set of positive integers  $BD = \{1, 2,
\ldots, \rho+1+m\}$, where $m$ is the cardinality of the basis
$A=\{a_1,\ldots ,a_m\}$ of the free group $F=F(A).$ We use letters
$i,j$, etc. for boundaries.

\item Two functions $\alpha : BS \rightarrow BD$ and $\beta : BS
\rightarrow BD$. We call $\alpha(\mu)$ and $\beta(\mu)$ the initial
and terminal boundaries of the base $\mu$ (or endpoints of $\mu$).
These functions satisfy the following conditions for every base $\mu
\in BS$: $\alpha(\mu) < \beta(\mu)$   if $\varepsilon (\mu)=1$ and
$\alpha(\mu) > \beta(\mu)$  if $\varepsilon (\mu)=-1$.
\item The set of boundary connections $(p,\lambda,q)$, where $p$ is
a boundary on $\lambda$ (between $\alpha(\lambda)$ and
$\beta(\lambda))$ and $q$ is a boundary on $\Delta (\lambda)$.
\end{enumerate}
\end{definition}

For a combinatorial generalized equation $\Omega$, one can
canonically associate a system of equations in {\bf variables} $h_1,
\ldots, h_\rho$ over $F(A)$ (variables $h_i$ are sometimes called
{\it items}). This system is called a {\bf generalized equation},
and (slightly abusing the language) we denote it by the same symbol
$\Omega$. The generalized equation  $\Omega$ consists of the
following three types of equations.
\begin{enumerate}
\item Each pair of dual  bases $(\lambda, \Delta(\lambda))$ provides an
equation $$[h_{\alpha(\lambda)} h_{\alpha(\lambda) + 1} \ldots
h_{\beta(\lambda ) - 1}]^ {\varepsilon (\lambda)} =$$ $$
[h_{\alpha(\Delta(\lambda))} h_{\alpha(\Delta (\lambda)) + 1} \ldots
h_{\beta(\Delta(\lambda )) - 1}]^{\varepsilon (\Delta(\lambda))}.$$
These equations are called {\bf basic equations}.

\item Every boundary connection $(p,\lambda,q)$ gives rise to a
\emph{boundary equation}
$$[h_{\alpha(\lambda)} h_{\alpha(\lambda) + 1} \cdots h_{p-1}] = [h_{\alpha(\Delta
(\lambda ))} h_{\alpha(\Delta(\lambda )) + 1} \cdots h_{q-1}],$$ if
$\varepsilon(\lambda) = \varepsilon(\Delta(\lambda))$ and
$$[h_{\alpha (\lambda )} h_{\alpha(\lambda ) + 1} \cdots h_{p-1}] = [h_{q}
h_{q+1} \cdots h_{\beta(\Delta(\lambda))-1}]^{-1} ,$$ if
$\varepsilon(\lambda)= -\varepsilon(\Delta(\lambda)).$
\item Constant equations: $h_{\rho+1}=a_1,\ldots
,h_{\rho+1+m}=a_m.$
\end{enumerate}

\begin{remark}
We  assume that every generalized equation comes associated with a
combinatorial one.
\end{remark}

Denote by $F_{R(\Omega )}$ the coordinate group of the generalized
equation.
\begin{definition}
Let $\Omega(h) = \{L_1(h) = R_1(h), \ldots, L_s(h) = R_s(h)\}$ be a
generalized equation in variables $h = (h_1, \ldots, h_{\rho})$. A
sequence of reduced nonempty words $U = (U_1(Z), \ldots,
U_{\rho}(Z))$ in the alphabet $(A\cup Z)^{\pm 1}$ is a {\em
solution} of $\Omega $ if:
\begin{enumerate}
\item all words $L_i(U), R_i(U)$ are reduced as written,
\item $L_i(U) =  R_i(U),\ i \in [1,s].$
\end{enumerate}
\end{definition}

If we specify a particular solution $\delta$ of a generalized
equation $\Omega$ then we use a pair $(\Omega, \delta )$.

It is convenient to  visualize  a generalized equation $\Omega$ as
follows.

\begin {center}
\begin{picture}(200,100)(0,0)
\put(10,85){\line(0,-1){80}} \put(30,85){\line(0,-1){80}}
\put(50,85){\line(0,-1){80}} \put(70,85){\line(0,-1){80}}
\put(90,85){\line(0,-1){80}} \put(110,85){\line(0,-1){80}}
\put(130,85){\line(0,-1){80}} \put(150,85){\line(0,-1){80}}
\put(170,85){\line(0,-1){80}} \put(10,80){\line(1,0){160}} \put
(10,87){$1$} \put (30,87){$2$} \put (50,87){$3$} \put (140,87){$\rho
-1$} \put (170,87){$\rho$} \put (10,65){\line(1,0){60}} \put
(70,45){\line(1,0){20}} \put (90,55){\line(1,0){40}} \put
(110,30){\line(1,0){40}} \put (20,67){$\lambda$} \put
(92,57){$\Delta(\lambda )$} \put (80,47){$\mu$} \put
(112,37){$\Delta(\mu)$}
\end{picture}
\end{center}

\subsection{ Elementary transformations}
\label{se:5.1}

In this section we describe {\em elementary transformations} of
generalized equations. Let $\Omega $ be a generalized equation  An
elementary transformation $(ET)$ associates to a generalized
equation $\Omega$ a family of generalized equations $ET(\Omega ) =
\{\Omega_1,\ldots ,\Omega _k\}$  and  surjective homomorphisms
$\pi_i : F_{R(\Omega)} \rightarrow F_{R(\Omega_i)}$ such that for
any solution $\delta$ of $\Omega$ and corresponding epimorphism $\pi
_{\delta}:F_{R(\Omega )}\rightarrow F$ there exists $i\in\{1,\ldots
,k\}$ and a solution $\delta _i$ of $\Omega _i$ such that the
following diagram commutes.

\[\begin{diagram}
\label{diag:1} \node{F_{R(\Omega )}} \arrow{e,t}{\pi _i}
\arrow{s,l}{\pi_{\delta}} \node{F_{R(\Omega_i)}}
\arrow{sw,r}{\pi_{\delta _i}} \\
\node{F}
\end{diagram}
\]

\begin{enumerate}
\item[(ET1)] {\em (Cutting a base (see Fig. 3))}. Let $\lambda$ be a base in
$\Omega$ and $p$ an {\em internal} boundary of $\lambda$ with a
boundary connection $(p,\lambda ,q).$ Then we cut the base $\lambda$
in $p$ into two new bases $\lambda_1$ and $\lambda_2$ and cut
$\bar\lambda$ in $q$ into the bases $\bar\lambda_1$, $\bar
\lambda_2$.

\begin{figure}[here]
\centering{\mbox{\psfig{figure=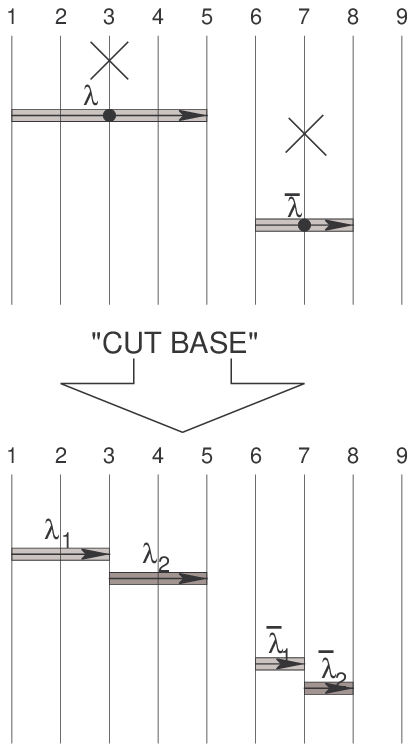}}} \caption{Elementary
transformation (ET1).} \label{ET1}
\end{figure}

\item[(ET2)] {\em (Transfering a base (see Fig. 4))}. If a base $\lambda$
of $\Omega$ contains a base $\mu$ (that is, $\alpha(\lambda) \leq
\alpha (\mu) < \beta(\mu) \leq \beta(\lambda)$) and all boundaries
on $\mu$ are $\lambda$-tied by boundary connections then we transfer
$\mu$ from its location on the base $\lambda$ to the corresponding
location on the base $\bar\lambda$.

\begin{figure}[here]
\centering{\mbox{\psfig{figure=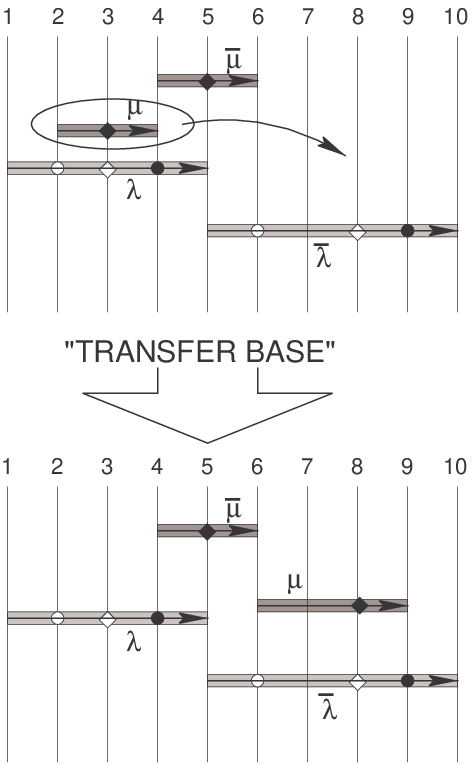}}} \caption{Elementary
transformation (ET2).} \label{ET2}
\end{figure}

\item[(ET3)] {\em (Removal of a pair of  matched bases (see Fig. 5))}. If
the bases $\lambda$ and $\bar\lambda$ are {\em  matched} (that is,
$\alpha(\lambda) = \alpha(\bar\lambda), \beta(\lambda) =
\beta(\bar\lambda)$) then we remove $\lambda, \bar\lambda$ from
$\Omega$.

\begin{remark}
Observe, that for $i = 1,2,3,$ $ ETi(\Omega)$ and $\Omega$ have the
same set of variables $H$, and the identity map $F[H] \rightarrow
F[H]$ induces an isomorphism $F_{R(\Omega)} \rightarrow
F_{R(\Omega')}$. Moreover, $\delta$ is a solution of $\Omega$ if and
only if $\delta$ is a solution of $\Omega'$.
\end{remark}

\begin{figure}[here]
\centering{\mbox{\psfig{figure=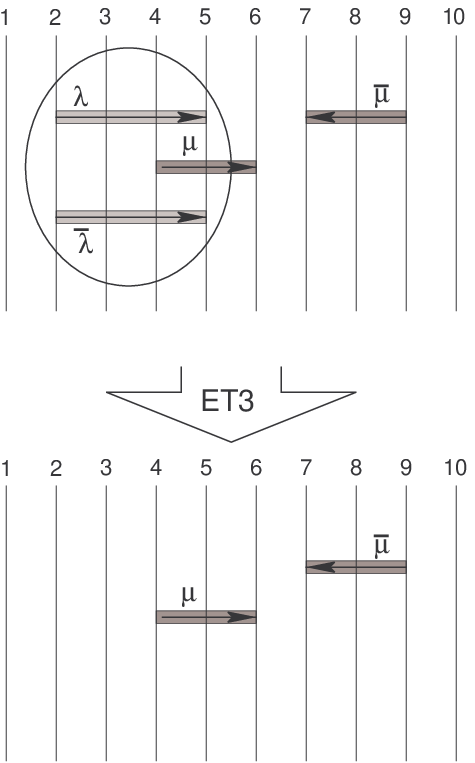}}} \caption{Elementary
transformation (ET3).} \label{ET3}
\end{figure}

\item[(ET4)] {\em (Removal of a lone base (see Fig. 6))}. Suppose, a base
$\lambda$ in $\Omega$ does not {\em intersect} any other base, that
is, the items $h_{\alpha(\lambda)}, \ldots, h_{\beta(\lambda)-1}$
are contained only on the  base $\lambda$. Suppose also that all
boundaries in $\lambda$ are $\lambda$-tied, i.e., for every $i$
($\alpha(\lambda) \leq i\leq \beta -1$) there exists a boundary
$b(i)$ such that $(i,\lambda ,b(i))$ is a boundary connection in
$\Omega$. Then we remove the pair of bases $\lambda$ and
$\bar\lambda$ together with all the boundaries $\alpha(\lambda)+1,
\ldots, \beta(\lambda)-1$ (and rename the rest $\beta(\lambda) -
\alpha(\lambda) - 1$ boundaries corespondignly).

We define the homomorphism $\theta: F_{R(\Omega)} \rightarrow
F_{R(\Omega')}$ as follows:
\[
\theta(h_j) = h_j\ {\rm if}\ j < \alpha(\lambda)\ {\rm or}\ j \geq
\beta(\lambda)
\]
\[
\theta(h_i) = \left\{\begin{array}{ll}
h_{b(i)} \ldots h_{b(i+1)-1}, & if\ \varepsilon(\lambda) = \varepsilon(\bar\lambda),\\
h_{b(i)}^{-1} \ldots h_{b(i+1)-1}^{-1},& if\ \varepsilon(\lambda) =
-\varepsilon(\bar\lambda)
\end{array}
\right.
\]
for $\alpha (\lambda) \leq i \leq \beta(\lambda)-1$. It is not hard
to see that $\theta$ is an isomorphism.

\begin{figure}[here]
\centering{\mbox{\psfig{figure=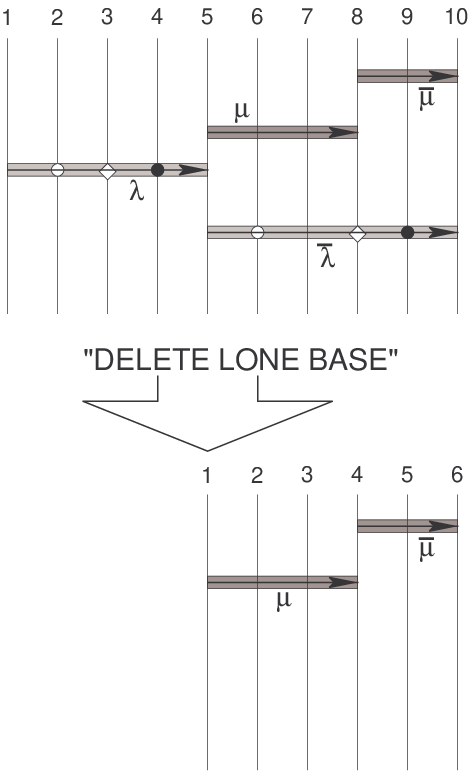}}} \caption{Elementary
transformation (ET4).} \label{ET4}
\end{figure}

\item[(ET5)] {\em (Introduction of a boundary (see Fig. 7))}. Suppose a
boundary $p$ in a base $\lambda$ is not $\lambda$-tied. The
transformation (ET5) $\lambda$-ties it. To this end, suppose
$\delta$ is a solution of $\Omega$. Denote $\lambda ^{\delta}$ by
$U_\lambda$, and let $U'_\lambda$ be the beginning of this word
ending at $p$. Then we perform one of the following transformations
according to where the end of $U'_\lambda$ on $\bar\lambda$ might be
situated:

\begin{enumerate}
\item If the end of $U'_\lambda$ on $\bar\lambda$ is situated on the boundary
$q$, we introduce the boundary connection $\langle p, \lambda ,q
\rangle$. In this case the corresponding homomorphism $\theta_q :
F_{R(\Omega)} \rightarrow F_{R(\Omega_q)}$ is induced by the
identity isomorphism on $F[H]$. Observe that $\theta_q$ is not
necessary an isomorphism.

\item If the end of $U'_\lambda$ on $\bar\lambda$ is situated
between $q$ and $q+1$, we introduce a new boundary $q'$ between $q$
and $q+1$ (and rename all the boundaries); introduce a new boundary
connection $(p,\lambda,q')$. Denote the resulting equation by
$\Omega _q'$. In this case the corresponding homomorphism
$\theta_{q'}: F_{R(\Omega)} \rightarrow F_{R(\Omega_{q'})}$ is
induced by the map $\theta_{q'}(h) = h$, if $h \neq h_q$, and
$\theta_{q'}(h_q) = h_{q'} h_{q'+1}$. Observe that $\theta_{q'}$ is
an isomorphism.
\end{enumerate}
Obviously, the is only a finite number of possibilities such that
for any solution $\delta$ one of them takes place.

\begin{figure}[here]
\centering{\mbox{\psfig{figure=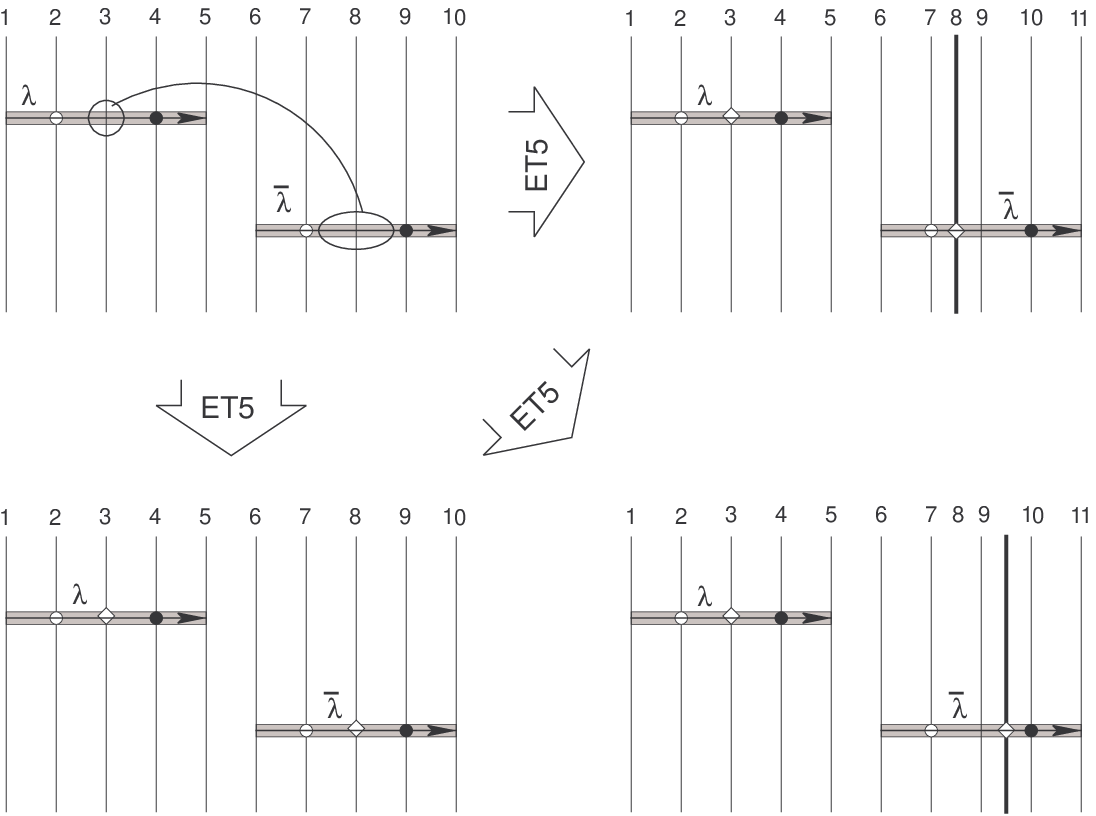}}} \caption{Elementary
transformation (ET5).} \label{ET5}
\end{figure}
\end{enumerate}

\subsection{Derived transformations and auxiliary transformations}
\label{se:5.2half}

In this section we define  complexity of a generalized equation and
describe several useful ``derived'' transformations of generalized
equations. Some of them can be realized as finite sequences of
elementary transformations, others result in equivalent generalized
equations but cannot be realized by finite sequences of elementary
moves.

A boundary is open if it is an internal boundary of some base,
otherwise it is closed. A section $\sigma =[i,\ldots ,i+k]$ is said
to be closed if boundaries $i$ and $i+k$ are closed and all the
boundaries between them are open.

Sometimes it will be convenient to subdivide all sections of
$\Omega$ into {\bf active} and {\bf non-active sections}. Constant
section will always be non-active.  A variable $h_q$ is called free
is it meets no base. Free variables are transported to the very end
of the interval behind all items in $\Omega$ and become non-active.
\begin{enumerate}

\item[(D1)] {\em (Deleting a complete base)}. A base $\mu$ of $\Omega$ is
called {\em complete} if there exists a closed section $\sigma$ in
$\Omega$ such that $\sigma = [\alpha(\mu), \beta(\mu)]$.

Suppose $\mu$ is a complete base of $\Omega$ and $\sigma$ is a
closed section  such that $\sigma = [\alpha(\mu),\beta(\mu)]$. In
this case using ET5, we transfer all bases from $\mu$ to $\bar\mu$;
using ET4, we remove the lone base $\mu$ together with the section
$\sigma(\mu)$.

{\bf Complexity.} Denote by $\rho $ the number of variables $h_i$ in
all (active) sections of $\Omega$, by $n = n(\Omega)$ the number of
bases in (active) sections of $\Omega$, by $n(\sigma )$ the number
of bases in a closed section $\sigma$.

The {\em complexity} of an equation $\Omega $ is the number
$$\tau = \tau (\Omega) = \sum_{\sigma \in A\Sigma_\Omega} max\{0,
n(\sigma)-2\},$$ where $A\Sigma_\Omega$ is the set of all active
closed sections.

\item[(D2)] {\it (Linear elimination)}.
 Let $\gamma (h_i)$ denote the number of bases met by $h_i$. A base $\mu
\in BS(\Omega )$ is called {\em eliminable} if at least one of the
following holds:
\begin{enumerate}
\item[(a)] $\mu$ contains an  item $h_i$ with $\gamma(h_i)=1$,

\item[(b)] at least one of the boundaries $\alpha (\mu),\beta (\mu)$
is different from $1,\rho +1 $, does not touch any other base
(except $\mu$) and is not connected by any boundary connection.
\end{enumerate}

We denote this boundary by $\epsilon$. A linear elimination for
$\Omega$ works as follows.

Suppose the base $\mu$ is removable because it satisfies condition
(b). We first cut $\mu$  at the nearest to $\epsilon$
$\mu$-connected boundary and denote it by $\tau$.  If there is no
such a boundary we denote by $\tau$ the other boundary of $\mu$.
Then we remove the base obtained from $\mu$ between $\epsilon$ and
$\tau$ together with its dual (maybe this part is the whole base
$\mu$), and remove the boundary $\epsilon$. Denote the new equation
by $\Omega '$.

%The coordinate group of the new generalized equation is isomorphic
%to $F_{R(\Omega)}$ but the number of items decreases.

Suppose the base is removable because it satisfies condition (a).

Suppose first that $\gamma (h_i)=1$ for  the leftmost item $h_i$ on
$\mu$. Denote by $\epsilon$ the left boundary of $h_i$. Let $\tau$
 be the nearest to $\epsilon$ $\mu$-connected boundary  (or the
 other terminal boundary of $\mu$ if there are no $\mu$-connected
 boundaries). We remove the base obtained from
$\mu$ between $\epsilon$ and $\tau$ together with its dual (maybe
this part is the whole base $\mu$), and remove $h_i$.

We make a mirror reflection of this transformation if $\gamma
(h_i)=1$ for the rightmost item $h_i$ on $\mu$.

Suppose now that $h_i$ is not the leftmost or the rightmost item on
$\mu$.
 Let $\epsilon$ and $\tau$ be the nearest to $h_i$ $\mu$-connected boundaries on the left
 and on the right of $h_i$ (each of them can be a terminal boundary of $\mu$).
 We cut $\mu$ at the boundaries $\epsilon$ and $\tau$, remove the
base between $\epsilon$ and $\tau$ together with its dual and remove
$h_i$.

\begin{lemma} Linear elimination does not increase
the complexity of $\Omega$, and the number of items decreases.
Therefore the linear elimination process stops after finite number
of steps.\end{lemma}

{\it Proof.} The input of the closed sections not containing $\mu$
into the complexity does not change. The section than contained
$\mu$ could be divided into two. In all cases except the last one
the total number of bases does not increase, therefore the
complexity cannot increase too. In the last case the number of bases
is increased by two, but the section is divided into two closed
sections, and each section contains at least two bases. Therefore
the complexity is the same. The number of items every time is
decreased by one.

We repeat linear elimination until no eliminable bases are left in
the equation. The resulting generalized equation is called a {\it
kernel} of $\Omega$ and we denote it by $Ker(\Omega)$. It is easy to
see that $Ker(\Omega)$ does not depend on a particular linear
elimination process. Indeed, if $\Omega$ has two different
eliminable bases $\mu_1, \mu_2$, and deletion of a part of $\mu_i$
results in an equation $\Omega_i$ then  by induction (on the number
of eliminations) $Ker(\Omega_i)$ is uniquely defined for $i = 1,2$.
Obviously, $\mu_1$ is still eliminable in $\Omega_2$, as well as
$\mu_2$ is eliminable in $\Omega_1$. Now eliminating $\mu_1$ and
$\mu_2$ from $\Omega_2$ and $\Omega_1$ we get one and  the same
equation $\Omega_0$. By induction $Ker(\Omega_1) = Ker(\Omega_0) =
Ker(\Omega_2)$ hence the result.

The following statement becomes obvious.
\begin{lemma} The generalized equation $\Omega$ (as a
system of equations over $F$) has a solution if and only if
$Ker(\Omega)$ has  a solution.
\end{lemma}
So linear elimination replaces $\Omega$ by $Ker(\Omega)$.

 Let us consider what happens on the group level in the process
of linear elimination. This is necessary only for the description of
all solutions of the equation.

We say that a variable $h_i$ {\it belongs to the kernel} ($h_i \in
Ker(\Omega)$), if either $h_i$ belongs to at least one base in the
kernel, or it is constant.

Also, for an equation $\Omega$ by $\overline{\Omega}$ we denote the
equation which is obtained from $\Omega$ by deleting all free
 variables. Obviously,
$$F_{R(\Omega)} = F_{R(\overline {\Omega})} \ast F(\bar Y)$$
where $\bar Y$ is the set of free variables in $\Omega$.

We start with the case  when a part of just one base is eliminated.
Let $\mu$ be an eliminable base in $\Omega = \Omega(h_1, \ldots,
h_\rho)$. Denote by $\Omega_1$ the equation resulting from $\Omega$
by eliminating $\mu$.

\begin{enumerate}
\item Suppose  $h_i \in \mu$ and $\gamma(h_i) = 1$. Let $\mu=\mu_1\ldots \mu_k$, where $\mu_1,\ldots ,\mu_k$ are the parts
between $\mu$-connected boundaries. Let $h_i\in\mu_j$. Replace the
basic equation corresponding to $\mu$ by the  equations
corresponding to $\mu_1,\ldots ,\mu_k$. Then the variable $h_i$
occurs only once in $\Omega$ - precisely in the equation $s_{\mu _j}
= 1$ corresponding to $\mu _j$. Therefore, in the coordinate group
$F_{R(\Omega)}$ the relation $s_{\mu _j}= 1$ can be written as $h_i
= w$, where $w$ does not contain $h_i$. Using Tietze transformations
we can rewrite the presentation of $F_{R(\Omega)}$ as
$F_{R(\Omega')}$, where $\Omega'$ is obtained from $\Omega$ by
deleting $s_{\mu _j}$ and the item $h_i$. It follows immediately
that
$$F_{R(\Omega_1)} \simeq  F_{R(\Omega')} \ast \langle h_i \rangle$$
and
\begin{equation}
\label{eq:ker1} F_{R(\Omega)} \simeq F_{R(\Omega')} \simeq
F_{R(\overline{\Omega_1})} \ast F(B)
\end{equation}
for some free or trivial group $F(B)$ .

\item Suppose now that $\mu$ satisfies  case b) above with respect to
a boundary $i$. Let $\mu=\mu_1\ldots \mu_k$. Replace the equation
$s_\mu = 1$ and the boundary equations corresponding to the boundary
connections through $\mu$ by the equations $s_{\mu_i}$, $i=1,\ldots
,k.$ Then in the equation $s_{\mu_k} = 1$ the variable $h_{i-1}$
either occurs only once or it occurs precisely twice and in this
event the second occurrence of $h_{i-1}$ (in $\bar\mu$) is a part of
the subword $(h_{i-1}h_i)^{\pm 1}$. In both cases it is easy to see
that the tuple
$$(h_1, \ldots, h_{i-2}, s_{\mu_k}, h_{i-1} h_i, h_{i+1}, \ldots,h_\rho)$$
forms a basis of the ambient free group generated by $(h_1, \ldots,
h_\rho)$ and constants from $A$. Therefore, eliminating the relation
$s_{\mu_k} = 1$, we can rewrite the presentation of $F_{R(\Omega)}$
in generators $\bar Y = (h_1, \ldots, h_{i-2}, h_{i-1} h_i, h_{i+1},
\ldots, h_\rho)$. Observe also  that any other basic  or boundary
equation $s_\lambda = 1$ ($\lambda \neq \mu$) of $\Omega$ either
does not contain variables $h_{i-1}, h_i$ or it contains them as
parts of the subword $(h_{i-1}h_i)^{\pm 1}$, that is, any such a
word $s_\lambda$ can be expressed as a word $w_\lambda(\bar Y)$ in
terms of generators $\bar Y$. This shows that
$$F_{R(\Omega)} \simeq G(\bar Y )_{R(w_\lambda(\bar Y) \mid \lambda \neq
\mu)} \simeq F_{R(\Omega')},$$ where $\Omega'$ is a generalized
equation obtained from $\Omega_1$ by deleting the boundary $i$.
Denote by $\Omega'$ an equation obtained from $\Omega'$ by adding a
free variable $z$ to the right end of $\Omega'$. It follows now that
$$F_{R(\Omega_1)} \simeq  F_{R(\Omega'')} \simeq F_{R(\Omega)} \ast
\langle z \rangle$$ and
\begin{equation}
\label{eq:ker2} F_{R(\Omega)} \simeq F_{R(\overline{\Omega'})} \ast
F(Z)
\end{equation}
or some free group $F(Z)$. Notice that all the groups and equations
which occur above can be found effectively.
\end{enumerate}

By induction on the number of steps in a cleaning process we obtain
the following lemma.
\begin{lemma}
\label{7-10} $F_{R(\Omega)} \simeq F_{R(\overline{Ker(\Omega)})}
\ast F(Z),$ where $F(Z)$ is a free group on $Z$.
\end{lemma}
{\it Proof.} Let
$$\Omega = \Omega_0 \rightarrow \Omega_1 \rightarrow \ldots \rightarrow
\Omega_l = Ker(\Omega)$$ be a linear elimination process for
$\Omega$. It is easy to see (by induction on $l$) that for every $j
\in [0,l-1]$
$$\overline{Ker(\Omega_j)} = \overline{Ker(\overline{\Omega_j})}.$$
Moreover, if $\Omega_{j+1}$ is obtained from $\Omega_j$ as in the
case 2 above, then (in the notation above)
$$\overline{Ker(\Omega_j)_1} = \overline{Ker(\Omega_j')} .$$
Now the statement of the lemma follows from the remarks above and
equalities (\ref{eq:ker1}) and (\ref{eq:ker2}).

\hfill $\square$
\end{enumerate}

\subsection{Rewriting process for $\Omega$}
\label{se:5.2}

In this section we describe a rewriting process for a generalized
equation $\Omega$.

\subsubsection{Tietze Cleaning and Entire Transformation}

In the rewriting process of generalized equations there will be two
main sub-processes:

 {\bf 1. Titze cleaning.} This process consists of repetition of the following four
transformations performed consecutively:
\begin{enumerate}
\item[(a)] Linear elimination,
\item[(b)] deleting all pairs of matched bases,
\item[(c)] deleting all complete bases,
\item[(d)] moving all free variables to the right and declare them non-active.
\end{enumerate}

{\bf 2. Entire transformation}.  This process is applied if $\gamma
(h_i)\geq 2$ for each $h_i$ in the active sections. We need a few
further definitions. A base $\mu$ of the equation $\Omega$ is called
a {\it leading} base if $\alpha(\mu) = 1$. A leading base is said to
be {\it maximal} (or a {\it carrier}) if $\beta(\lambda) \leq
\beta(\mu),$ for any other leading base $\lambda $. Let $\mu $ be a
carrier base of $\Omega.$ Any active base $\lambda \neq \mu$ with
$\beta(\lambda) \leq \beta(\mu)$ is called a {\it transfer} base
(with respect to $\mu$).

Suppose now that $\Omega$ is a generalized equation with
$\gamma(h_i) \geq 2$ for each $h_i$ in the active part of $\Omega$.
An {\em entire transformation} is a sequence of elementary
transformations which are performed as follows. We fix a carrier
base $\mu$ of $\Omega$. We transfer all transfer bases from $\mu$
onto $\bar\mu$. Now, there exists some $i < \beta (\mu)$ such that
$h_1,\ldots ,h_i$ belong to only one base $\mu$, while $h_{i+1}$
belongs to at least two bases. Applying (ET1) we cut $\mu$ along the
boundary $i+1$. Finally, applying (ET4) we delete the section
$[1,i+1]$.

Notice that neither process increases complexity.

\subsubsection{Solution tree}
 Let $\Omega $ be a
generalized equation. We construct a solution tree  $T(\Omega)$
(with associated structures), as a rooted tree oriented from the
root $v_0$, starting at $v_0$ and proceeding by induction on the
distance $n$ from the root.

If
$$v \rightarrow v_1 \rightarrow \cdots \rightarrow v_s \rightarrow u$$
is a path in $T(\Omega )$, then by $\pi (v,u)$ we denote composition
of corresponding epimomorphisms
$$\pi(v,u) = \pi(v,v_1) \cdots \pi(v_s,u).$$

If $v \rightarrow v'$ is an edge then there exists a finite sequence
of elementary or derived transformations from $\Omega_v$ to
$\Omega_{v'}$ and the homomorphism $\pi(v,v')$ is composition of the
homomorphisms corresponding to these transformations. We also assume
that active [non-active] sections in $\Omega_{v'}$ are naturally
inherited from $\Omega_v$, if not said otherwise.

Suppose a path in $T(\Omega)$ is constructed by induction up to a
level $n$, and suppose $v$ is a vertex at distance $n$ from the root
$v_0$. We describe now how to extend the tree from $v$.

We apply the Tietze cleaning at the vertex $v_n$ if it is possible.
If it is impossible ($\gamma (h_i)\geq 2$ for any $h_i$ in the
active part of $\Omega _{v}$), we apply the entire transformation.
Both possibilities  involve either creation of new boundaries and
boundary connections or creation of new boundary connections without
creation of new boundaries, and, therefore, addition of new
relations to $F_{R(\Omega _{v})}$. The boundary connections can be
made in few different ways, but there is a finite number of
possibilities. According to this, different resulting generalized
equations are obtained, and we draw edges from $v$ to all the
vertices corresponding to these generalized equations.

{\bf Termination condition:} 1. $\Omega_v$ does not contain active
sections. In this case the vertex $v$ is called a {\it leaf} or an
{\it end vertex}. There are no outgoing edges from $v$.

2. $\Omega _v$ is inconsistent. There is a base $\lambda$ such that
$\bar\lambda $ is oriented the opposite way and overlaps with
$\lambda$, or the equation implies an inconsistent constant
equation.

\subsubsection{Quadratic case}

Suppose $\Omega_v$ satisfies the condition $\gamma _i = 2$ for each
$h_i$ in the active part. Then $F_{R(\Omega _v)}$ is isomorphic to
the free product of a free group and a coordinate group of a
standard quadratic equation (to be defined below) over the
coordinate group $F_{R(\Omega ')}$ of the equation $\Omega '$
corresponding to the non-active part. In this case entire
transformation can go infinitely along some path in $T(\Omega )$,
and, since the number of bases if fixed, there will be vertices $v$
and $w$ such that $\Omega _v$ and $\Omega _w$ are the same. Then the
corresponding epimorphism $\pi :F_{R(\Omega _v)}\rightarrow
F_{R(\Omega _w)}$ is an automorphism of $F_{R(\Omega _v)}$ that
decreases the total length of the interval. For a minimal solution
of $F_{R(\Omega _v)}$ the process will stop.

\begin{definition}
A standard quadratic equation over the group $G$ is an equation of
the one of the following forms (below $d,c_i$ are nontrivial
elements from $G$):
\begin{equation}\label{eq:st1}
\prod_{i=1}^{n}[x_i,y_i] = 1, \ \ \ n > 0;
\end{equation}
\begin{equation}\label{eq:st2}
\prod_{i=1}^{n}[x_i,y_i] \prod_{i=1}^{m}z_i^{-1}c_iz_i d = 1,\ \ \
n,m \geq 0, m+n \geq 1 ;
\end{equation}
\begin{equation}\label{eq:st3}
\prod_{i=1}^{n}x_i^2 = 1, \ \ \ n > 0;
\end{equation}
\begin{equation}\label{eq:st4}
\prod_{i=1}^{n}x_i^2 \prod_{i=1}^{m}z_i^{-1}c_iz_i d = 1, \ \ \ n,m
\geq 0, n+m \geq 1.
\end{equation}

Equations (\ref{eq:st1}), (\ref{eq:st2}) are called {\em
orientable}, equations (\ref{eq:st3}), (\ref{eq:st4}) are called
{\em non-orientable}. Number $n$ is called a {\em genus} of the
equation (notation $gen (S).$)
\end{definition}

The proof of the following fact can be found in \cite{LS}.

\begin{lemma} \label{EC1}
Let $W$ be a strictly quadratic word over $G$. Then there is a $G$-
automorphism $f \in Aut_G(G[X])$ such that  $W^f$ is a standard
quadratic word over $G.$
\end{lemma}

\subsubsection{Entire transformation goes infinitely}

Let now $\gamma (h_i)\geq 2$ for all $h_i$ in the active part, and
for some $h_i$ the inequality is strict. Let $(\Omega ,\delta )$ be
a generalized equation with a solution.  Define the {\bf excess}
$\psi$ of $(\Omega ,\delta ) $:

$$\psi =\Sigma _{\lambda} (\lambda ^{\delta})-2|I^{\delta}|,$$
where $\lambda$ runs through the set of bases participating in
entire transformation and $I^{\sigma}$ is the segment between the
initial point of the interval and the leftmost point of the base
that never participates (as carrier or transfer base).

It is possible that the cleaning after the entire transformation
decreases complexity. This occurs if some base is transferred onto
its dual and removed by (ET3). Otherwise, we use the same name for a
base of $\Omega _i$ and the reincornation of this base in $\Omega
_{i+1}.$ If we cannot apply Tietze cleaning after the entire
transformation, then we successively apply entire transformation. It
is possible that the entire transformation sequence for $\Omega $
goes infinitely, and the complexity does not decrease. If we apply
the entire transformation to $(\Omega ,\delta )$ and the complexity
does not decrease, then $\psi$ does not change.

We say that  bases $\mu$ and its dual of the equation $\Omega$ form
an overlapping pair  if $\mu$ intersects with its dual $\bar\mu$.
%and is at least twice longer than $|\alpha(\bar\mu) - \alpha(\mu)|$.

If $\phi _1$ and $\phi_2$ are two solutions of a generalized
equation $\Omega$ in $F(A,Y)$, then we define $\phi _1<\phi _2$ if
$\phi _2=\sigma\phi _1\pi$, where $\sigma$ is a canonical
automorphism of $F_{R(\Omega )}$ and $\pi$ is an endomorphism of
$F(A,Y)$, and $\sum _{i=1}^{\rho}(h_i)^{\phi_1}<\sum_{i=1}^{\rho}
(h_i)^{\phi_2}$. Then we can define minimal solutions of a
generalized equation.

\begin{theorem} \cite{KMIrc} Let $(\Omega, \delta )$ be a
generalized equation with a minimal solution. Suppose $(\Omega,
\delta )=(\Omega _0,\delta _0),\  (\Omega _1,\delta _1), \ldots $ be
the generalized equations (with solutions) formed by the entire
transformation sequence. Then one can construct a number $N=N(\Omega
)$ such that the sequence ends after at most $N$ steps.\end{theorem}

We will prove  the key lemmas.
\begin{lemma} If $\delta$ is a solution minimal with respect to the
subgroup of $\mathcal A$ generated by the canonical Dehn twists
corresponding to the quadratic part of $\Omega$, then one can
construct a recursive function $f=f(\Omega )$ such that
$|I^{\delta}|\leq f\psi.$
\end{lemma}
This lemma shows that for a minimal solution the length of the
participating part of the interval is bounded in terms of the
excess. And the excess does not change in the sequence of entire
transformations when the complexity does not decrease.

\begin{proof} We can temporary change generalized equation
$\Omega$  such a way that it consists of one or several quadratic
closed sections (such that $\gamma (h_i)=2$ for any $h_i$) and
non-quadratic sections (such that  $\gamma (h_i)>2$ for any $h_i$).
Indeed, if $\sigma$ is a quadratic section of $\Omega$, we can cut
all bases in $\Omega$ through the end-points of $\sigma$. Moreover,
we will put the non-quadratic sections on the right part of the
interval. Denote by $\Omega_1$ this new generalized equation. We
apply the entire transformation to the pair ($\Omega _1, \delta
_1)$, where $\delta _1$ is obtained from $\delta$, and, therefore,
minimal. We can find a number $k(\Omega )$ such that after $k$
transformations
$$(\Omega _1,\delta _1)\rightarrow\ldots\rightarrow (\Omega _k,\delta _k)$$
all bases situated on the quadratic part will either form matched
pairs or will be transferred to the non-quadratic part. Indeed,
while we transforming the quadratic part we notice that:

1) two equations $\Omega _i$ and $\Omega _j$ for $i<j$ cannot be the
same, because then $\delta _j$ would be shorter than $\delta _i$,
contradicting the minimality.

2) there is only a finite number of possibilities for the quadratic
part since the number of items and complexity does not increase.

The sequence of consecutive quadratic carrier bases is bounded.
Therefore after a bounded number of steps, a quadratic coefficient
base is carrier, and we transfer a transfer base to the
non-quadratic part. For a minimal solution, the length of a free
variable corresponding to a matching pair is 1. And for each base
$\lambda$ transfered to the non-quadratic part, $\lambda ^{\delta}$
is shorter than the interval corresponding to the non-quadratic
part, and, therefore, shorter than $\psi$. This gives a hint how to
compute a function $f(\Omega)$. We can now return to the generalized
equation $\Omega$ and replace its solution by a minimal solution
$\delta $. \end{proof}

The {\bf exponent of periodicity} of a family of reduced words
$\{w_1,\ldots ,w_k\}$ in a free group $F$ is the maximal number $t$
such that some $w_i$ contains a subword $u^t$ for some simple
cyclically reduced word $u$. The exponent of periodicity of a
solution $\delta$ is the exponent of periodicity of the family
$\{h_1^{\delta},\ldots ,h_{\rho}^{\delta}\}.$

We call a solution of a system of equations in the group $F(A,Y)$
{\em strongly minimal} if it is minimal and cannot be obtained from
a shorter solution by a substitution.

\begin{lemma} (Bulitko's lemma).
 Let $S$ be a system of equations over a free group. The exponent
of periodicity of a strongly minimal solution can be effectively bounded.
\end{lemma}

{\em Proof.}  Let $P$ be a simple cyclically reduced word. A {\em
$P$-occurrence} in a word $w$ is an occurrence in $w$ of a word
$P^{\ep t}$, $\ep=\pm1$, $t\ge1$. We call a $P$-occurrence $v_1\cdot
P^{\ep t} \cdot v_2$ {\em stable} if $v_1$ ends with $P^\ep$ and
$v_2$ starts with $P^\ep$. Clearly, every stable $P$-occurrence lies
in a maximal stable $P$-occurrence. Two distinct maximal stable
$P$-occurrences do not intersect.

A {\em $P$-decomposition ${\mathcal D}_P(w)$ of a word $w$} is the
unique representation of $w$ as a product
$$
  v_0 \cdot P^{\ep_1 r_1} \cdot v_1 \cdot \ldots \cdot P^{\ep_m r_m} \cdot v_m
$$
where the occurrences of $P^{\ep_i r_i}$ are all maximal stable
$u$-occurrences in $w$. If $w$ has no stable $P$-occurrences then,
by definition, its $P$-decomposition is trivial, that is, it has one
factor which is $w$ itself.

By adding new variables we can transform  the system $S$ is the
triangular form, namely, such that each equation has length 3. If we
have equation $xyz=1$ with solution $x^{\phi}, y^{\phi}, z^{\phi}$,
then the cancellation table for this solution looks as the triangle
 in Fig. 2.

Let $$
 x^{\phi}= v_{10} \cdot P^{\ep_{11} r_{11}} \cdot v_{11} \cdot \ldots \cdot P^{\ep_{1,m} r_{1,m}} \cdot
 v_{1,m},
$$
$$
 y^{\phi}= v_{20} \cdot P^{\ep_{21} r_{21}} \cdot u_{21} \cdot \ldots \cdot P^{\ep_{2,n} r_{2,n}} \cdot
 v_{2,n},
$$

$$
 z^{\phi}= v_{30} \cdot P^{\ep_{31} r_{31}} \cdot v_{31} \cdot \ldots \cdot P^{\ep_{3,k} r_{3,k}} \cdot v_{3,k}
$$
be corresponding $P$-decompositions. From the cancellation table we
will have a system of equations on the natural numbers $r_{ij},\
i=1,2,3, j=1,\ldots max\{k,m,n\}.$ All equations except, maybe, one
will have form $r_{ij}=r_{st}$ for some pairs $i,j$ and $s,t$ and
one equation may correspond to the middle of the triangle. If the
middle of the triangle is inside a stable $P$-occurrence in
$z^{\phi}$, then the equation would either have form
$r_{1j}+r_{2s}+2=r_{3t}$ or $r_{1j}+r_{2s}+3=r_{3t}$. Notice that
since $x^{\phi}, y^{\phi}, z^{\phi}$ are reduced words, the middle
of the triangle cannot be inside a stable $P$-occurrence for more
than one variable.

If we replace a solution $r_{ij},\ i=1,2,3, j=1,\ldots
,max\{k,m,n\}$ of this system of equations by another positive
solution, say $q_{ij},\ i=1,2,3, j=1,\ldots ,max\{k,m,n\}$ and
replace in the solution $x^{\phi}, y^{\phi}, z^{\phi}$ stable
$P$-occurrences $P^{r_{ij}}$ by $P^{q_{ij}}$ we will have another
solution of the equation $xyz=1$.

Now, instead of one equation $xyz=1$ we take a system of equations
$S$. We obtain a corresponding linear system for natural
numbers $r_{ij}$'s. Let $R$ be the family of variables $r_{ij}$'s
that occur in the linear equations of length 3.  The number of such
equations is not larger than the number of triangles, that is the
number of equations in the system $S$. Therefore $R$ is a finite
family. Consider a system of all linear equations on $R$. It depends
on the particular solution of $S$, but there is a finite number of
possible such systems. We now can replace values of variables from
$R$ by a minimal positive solution, say $\{q_{ij}\}$, of the same
linear system (if $r_{ij}$ does not appear in any linear equation we
replace it by $q_{ij}=1$) and replace in the solution of the system
$S$ stable $P$-occurrences $P^{r_{ij}}$ by $P^{q_{ij}}$. We obtain
another solution of the system $S$. The length of a minimal positive
solution $\{q_{ij}\}$ of the linear system is bounded as in the
formulation of the lemma. The lemma is
proved.

\begin{lemma}
Suppose $F_{R(\Omega )}$ is not a free product with an abelian
factor, and there are solutions of $\Omega$ with unboundedly large
exponent of periodicity.   One can effectively find a number $M$ and
an abelian splitting of  $F_{R(\Omega )}$  (or a quotient obtained
from $F_{R(\Omega )}$  by adding commutation transitivity condition
for certain subgroups) as an amalgamated product with abelian vertex
group or as an HNN-extension (or both), such that the exponent of
periodicity of a minimal solution of $\Omega$ with respect to the
group of canonical automorphisms corresponding to this splitting and
the quadratic part (if exists) is bounded by $M$.
\end{lemma}

The proof of this lemma uses the notion of a periodic
structure and can be found in (\cite{Imp}, Lemma 22) or in
\cite{KMIrc}.

 Consider an infinite path in $T(\Omega)$ corresponding to an
infinite sequence in entire transformation
\begin{equation}
\label{3.6} r = v_1 \rightarrow v_2 \rightarrow \cdots \rightarrow
v_m.
\end{equation}

 Let $\delta$ be a
solution of $\Omega $. The following lemma gives the way to
construct a function $f_1$ depending on $\Omega $ such that for any
number $M$, if the sequence of entire transformations for $(\Omega
,\delta )$ has $f_1(M)$ steps, then $|I^{\delta}|>M\psi .$

Denote by $\mu_i$ the carrier base of the equation $\Omega_{v_i}$.
The path (\ref{3.6}) will be called $\mu$-reducing if $\mu_1 = \mu$
and one of the following holds:

1. $\mu_2$ does not overlap with its double and $\mu$ occurs in the
sequence $\mu_1, \ldots, \mu_{m-1}$ at least twice.

2.  $\mu_2$  overlaps with its double and $\mu$ occurs in the
sequence \newline $\mu_1, \ldots, \mu_{m-1}$ at least $M+2$ times,
where $M$ is the exponent of periodicity of $\delta$.

The following lemma is just an easy exercise.
\begin{lemma} In a $\mu$-reducing path the length of $I^{\delta}$ decreases at
least by $|\mu ^{\delta}|/10$.\end{lemma}
\begin{proof} Case 1. $\mu=\mu _1\neq \mu _2$, and not more than half of
$\mu _2$ overlaps with its double. Then after two steps the leftmost
boundary of the reincornation of $\mu$ will be to the right of the
middle of $\mu _2$. Therefore by the time when the reincornation of
$\mu$ becomes a carrier, the part from the beginning of the interval
to the middle of $\mu _2$ will be cut and removed. This part is
already longer than half of $\mu .$

Case 2.  $\mu=\mu _1= \mu _2$, and $\mu _2$ (second reincornation of
$\mu$) does not overlap with its double. Then on the first step we
cut the part of the interval that is longer than half of $\mu$.

Case 3. $\mu _2$ overlaps with its double. Denote by $\mu ^{\delta
(i)}$ the value of the  reincornation of $\mu^{\delta}$ on step $i$
and by $[1,\sigma]^{\delta (i)}$ the word corresponding to the
beginning of the interval until boundary $\sigma$ on step $i$. Then
$[1,\alpha (\bar \mu _2))]^{\delta (2)}=P^d$ for some cyclically
reduced word $P$ which is not a proper power and  $\mu ^{\delta
(2)}$,  $\mu_2 ^{\delta (2)}$ are beginnings of $[1,\beta
(\bar\mu_2)]^{\delta (2)}$ which is a beginning of $P^{\infty}$.

 We have
 \begin{equation}\label{1}
 \mu ^{\delta (2)}=P^rP_1 , r\leq
 M\end{equation}

Let $\mu _{i_1}=\mu_{i_2}=\mu$ for $i_1<i_2$ and $\mu _i\neq\mu$ for
$i_1<i<i_2$. If \begin{equation}\label{2} | \mu _{i_1+1}^{\delta
(i_1+1)}|\geq 2|P|\end{equation} and $[1,\rho _{i_1+1}+1]^{\delta
(i_1+1)}$ begins with a cyclic permutation of $P^3$, then
$$|[1,\alpha(\bar\mu_{i_1+1})]^{\delta (i_1+1}|\geq |P|.$$
The base $\mu$ occurs in the sequence $\mu _1,\ldots ,\mu_{m-1}$ at
least $r+1$ times, so either (\ref{2}) fails for some $i_1\leq m-1$
or the part of the interval that was removed after $m-1$ steps is
longer than $max\{|r-3||P|,|P|\}$.

If (\ref{2}) fails, then $|[1,\alpha (\mu _{i_1})]^{\delta
(i_1)}|\geq (r-2)|P|.$ So everything is reduced to the case when the
part of the interval that was removed after $m-1$ steps is longer
than $max\{|r-3||P|,|P|\}$. Together with (\ref{1}) this implies
that in $m-1$ steps the length of the interval was reduced at least
by $\frac{1}{5}|\mu ^{\delta (2)}|$ which is not less than
$\frac{1}{10}|\mu^{\delta}|.$
\end{proof}

We can now finish the proof of Theorem 9. Let $L$ be the family of
bases such that every base $\mu\in L$  occurs infinitely often as a
leading base. Suppose a number $m$ is so big that for every base
$\mu$ in $L$, a $\mu$-reducing path occurs more than $20nf$ times
during these $m$ steps. Since $\Sigma |\mu ^{\delta _m}|\geq \psi $,
where we sum over all the participating bases, at least for one base
$\lambda\in L$ , $|\lambda ^{\delta _m}|\geq \psi /2n$. Moreover,
$|\lambda ^{\delta _i}|\geq |\lambda ^{\delta _m}|\geq \psi /2n$ for
all $i\leq m$. Since a $\lambda$-reducing path occurs more than
$20nf$ times, the length of the interval would be decreased in $m$
steps by more than it initially was. This gives a bound on the
number of steps in the entire transformation sequence for $(\Omega
,\delta )$ for a minimal solution $\delta$. Theorem 9 has been
proved.

We will now prove Theorem 4. We replace in the tree $T(\Omega )$
every infinite path corresponding to an infinite sequence of entire
transformation of generalized equations beginning at $\Omega _{v_i}$
by a loop corresponding to automorphisms of $F_{R(\Omega _{v_i})}$
and  finite sequence of transformations for a minimal solution of
$\Omega _{v_i}$.  At the end of this sequence of transformations we
obtain a generalized equation $\Omega _{v_j}$ such that either it
has smaller complexity than $\Omega _{v_i}$ or $F_{R(\Omega
_{v_j})}$ is a proper quotient of $F_{R(\Omega _{v_i})}$. Any proper
chain of residually free quotients is finite. Therefore we obtain a
finite graph (the only cycles are loops corresponding to
automorphisms) and its maximal subtree . The equation $S(X)=1$ has a
non-trivial solution if and only if we are able to construct
$T_{sol}(\Omega )$ at least for one of the generalized equations
corresponding to the system $S(X)=1$.

Let $v_0\rightarrow v_1\rightarrow \ldots v_k$ be a path in
$T_{sol}(\Omega )$ from the root to a leaf. Let $v_{i-1}$ ($i\geq
1$) be the first vertex such that there is a loop corresponding to
automorphisms of $F_{R(\Omega _{v_{i-1}})}$ attached to $v_{i-1}$.
And let   $v_{j}$ be the next such vertex or (if there is no such a
second vertex) $v_j=v_k$. All the homomorphisms from $F_{R(\Omega
)}$ to $F$ in the fundamental set corresponding to the path from
$v_0$ to $v_k$ factor through a free product of $F_{R(\Omega
_{v_{i-1}})}$ and, maybe, some free group (that occured as a result
of Titze cleaning when going from $\Omega _{v_0}$ to $\Omega
_{v_{i-1}}$). All the homomorphisms from $F_{R(\Omega _{v_{i-1}} )}$
to $F$ in the fundamental set corresponding to the path from
$v_{i-1}$ to $v_k$ are obtained by the composition of a canonical
automorphism $\sigma$ of $F_{R(\Omega _{v_{i-1}})}$, canonical
epimorphism $\pi=\pi _i\ldots\pi _j$ from $F_{R(\Omega _{v_{i-1}})}$
onto $F_{R(\Omega _{v_j})}$ and a homomorphism from the fundamental
set of homomorphisms from $F_{R(\Omega _{v_j})}$ to $F$. The
composition $\sigma\pi $ is a solution of some system of equations,
denoted by $S_1(H_1,H_2,H_3,H^{\pi},A)=1$, over $F_{R(\Omega
_{v_j})}$. (Notice that by $H$ we denote a generating set of
$F_{R(\Omega _{v_{i-1}})}$ modulo $F(A)$). Therefore $H^{\pi}$ and
$A$ are the sets of coefficients of this system.) The system
$S_1(H_1,H_2,H_3,H^{\pi},A)=1$ consists of three types of
subsystems:

1. Quadratic system in variables $h\in H_1$, where $H_1$ is the
collection of items in the quadratic part of $\Omega _{v_{i-1}}$.
This system is obtained from $\Omega$ by replacing in each basic,
and boundary equation each variable $h$ in the non-quadratic part by
the coefficient $h^{\pi}$.

2. For each splitting of $F_{R(\Omega _{v_{i-1}})}$ as an
amalgamated product with a free  abelian vertex group of rank $ k$
from the second part of Lemma 6, we reserve $k$ variables
$x_1,\ldots ,x_k\in H_2$ and write commutativity equations
$[x_i,x_j]=1$ for $i,j=1,\ldots,k$ and, in addition, equations
$[x_i,u^{\pi}]=1$ for each generator $u$ of the edge group.

 3. For each variable $x\in
H_3$ there is an equation $xu^{\pi}x^{-1}=v^{\pi}$, where $u$ is a
generator of an edge group corresponding to a splitting of
$F_{R(\Omega _{v_{i-1}})}$ as an HNN-extension from the third part
of Lemma 6 and $x$ corresponds to the stable letter of this
HNN-extension.

Notice that $\sigma\pi$ is a solution of the system
$S_1(H_1,H_2,H_3,H^{\pi},A)=1$  over the group $F_{R(\Omega
_{v_j})}$.

4. For each $x\in H_3$ and corresponding edge group we introduce a
new variable $y$ and equations $[y,u^{\pi}]=1$, where $u$ is a
generator of the edge group. Let $H_4$ be the family  of these new
variables.

Denote by $S_2(H_1,H_2,H_4,H^{\pi},A)=1$ the system of equations 1,2
and 4. This system is NTQ over $F_{R(\Omega _{v_j})}$. Make a
substitution $x=x^{\pi}y$. Using this substitution every solution
$\sigma\pi$ of the system $S_1(H_1,H_2,H_3,H^{\pi},A)=1$ in the
group $F_{R(\Omega _{v_j})}$ can be obtained from a solution of the
NTQ system $S_2(H_1,H_2,H_4,H^{\pi},A)=1$.

We made the induction step. Since $T_{sol}(\Omega )$ is finite, the
proof of Theorem 4 can be completed by induction.

It is clear that Theorem 9 is the main technical result required for
the proof of Theorem 8 or, in other terminology, the  base for our
``shortening argument''. The proof of Theorem 9 is technically
complicated because everything is done effectively (the algorithms
are given). For comparison we will give a non-constructive proof of
Theorem 9 using the following lemma.

\begin{lemma} Let $\Omega _0,\ \Omega _1, \ldots $ be the
generalized equations formed by the entire transformation sequence.
Then one of the following holds. \begin{enumerate} \item the
sequence ends, \item for some $i$ we obtain the quadratic case on
the interval $I$, \item we obtain an overlapping pair $\lambda $,
$\bar\lambda $ such that $\lambda$ is a leading base, $\lambda
^{\delta}$ begins with some $n$'th power of the word $[\alpha
(\lambda),\alpha (\bar\lambda )]^{\delta }$ and there are solutions
${\delta}$ of $\Omega _i$ with number $n$ arbitrary large (with
arbitrary large exponent of periodicity).
\end{enumerate}\end{lemma}

\begin{proof}  We assume cases 1 and 2 do not hold. Then, our
sequence is infinite and we may assume that every base that is
carried is carried infinitely often, and that every base that
carries does so infinitely often. So every base that participates
does so infinitely often. We also assume the complexity does not
change and that no base is moved off the interval.

Let $\Omega$ be a generalized equation (with solution) and let
$$B=\Omega_1,\ldots ,\Omega_n,\ldots $$ be an infinite branch. Let
$\delta_1,\delta_2,\ldots $ be a set of solutions of $\Omega$ such
that $\delta_i$ "factors" through $\Omega_i$.

If we rescale the metric so that $I^{\delta_i} $ has length 1, then
each $\delta_i $ puts a length function on the items of $\Omega,$ in
particular we assume that each base has length and midpoint between
0 and 1. This means that for each $\delta_i$ there is a point $x_i$
in $[0,1]^m$, where this point represents the lengths of the bases
and the items as well as their midpoints in the normalized metric.

We pass to a subsequence of $\delta_i$ (omit the double subscript)
such that the $ x_i$ converge to a point $x$ in $[0,1]^m$. We call
the limit {\em a metric  on} $(\Omega,$ and denote it $\delta^*)$.

Normed excess denoted $m(\psi )$ is a constant and we can apply the
Bestvina, Feighn argument (toral case) on the generalized equation
$\Omega$ with lengths given by $\delta^*$.

The argument goes as follows.  Entire transformation is moving bases
to the right and shortening them, and $m(\psi )$ is a constant.
During the process the initial point of every base is only moved
towards  the final point of $I$, and the length of a base is never
increased, therefore, every base has a limiting position. Since
$m(\psi )$ is a constant, there is a base $\lambda$ of length not
going to zero that participates infinitely often. If $\lambda$ is
eventually the only carrier, then we must have case 3 for the
process to go unboundedly long. Suppose $\lambda$ is carried
infinitely often. Whenever $\lambda$ is the carrier, the midpoint of
some base moves the distance between the midpoints of $\lambda$ and
its dual. Since every base has a limiting position, it follows that
$\lambda$ and its dual have the same limiting position.

The argument shows that after some finite number of steps we get an
overlapping initial section i.e. carrier and dual have high length,
but midpoints are close.

 It follows that for n sufficiently large doing the process with
$(\Omega,\delta_n)$ will give a similar picture.

 This implies case 3.

\end{proof}

Case 3 can only happen is there are solutions of an arbitrary large
exponent of periodicity. If we consider only minimal solutions, then
the exponent of periodicity can be effectively bounded, and entire
transformation always stops after a bounded number of steps.

On the group level, case 2 corresponds to the existence a QH vertex
group in the JSJ decomposition of $F_{R(\Omega )}$ and case 3
corresponds to the existence of an abelian vertex group in the
abelian JSJ decomposition of $F_{R(\Omega )}.$

\section{Elementary free groups}
If an NTQ group does not contain non-cyclic abelian subgroups we
call it {\em regular NTQ group}. We have shown in \cite{KMIrc} that
regular NTQ groups are hyperbolic. (Later Sela called these groups
{\em hyperbolic $\omega$-residually free towers} \cite{Sela1}.)
\begin{theorem}\cite{KM4},\cite{Sela6} Regular NTQ groups are
exactly the f.g. models of the elementary theory of a non-abelian
free group.\end{theorem}

\section{Stallings foldings and algorithmic problems}
A new technique to deal with $F^{\mathbb{Z}[t]}$ became available
when Myasnikov, Remeslennikov, and Serbin showed that elements of
this group can be viewed as reduced {\it infinite} words in the
generators of $F$.  It turned  out that many algorithmic problems
for  finitely generated fully residually free groups can be solved
by the same methods as in the standard free groups. Indeed, they
introduces an analog of the Stallings' folding  for an arbitrary
finitely generated subgroup of $F^{\mathbb{Z}[t]}$, which allows one
to solve effectively the membership  problem in $F^{\mathbb{Z}[t]}$,
as well as in an arbitrary finitely generated subgroup of it.

\begin{theorem} (Myasnikov-Remeslennikov-Serbin)\cite{MRS2} Let $G$ be a f.g.
fully residually free group and $G\hookrightarrow G^\ast$ the
effective Nielsen completion. For any f.g. subgroup $H \leq G$ one
can effectively construct a finite graph $\Gamma_H$ that in the
group $G^\ast$ accepts precisely the normal forms of elements from
$H$.
\end{theorem}

\begin{theorem} \cite{KMRS}  The following algorithmic problems are
decidable in a f.g. fully residually free group G:
  \begin{itemize}
\item the membership problem,
\item the intersection problem (the intersection
   of two f.g. subgroups in $G$ is f.g. and one can find a finite generated set
   effectively),
  \item conjugacy of f.g. subgroups,
   \item malnormality of subgroups,
   \item finding the centralizers of finite subsets.
\end{itemize}\end{theorem}

It was proved by Chadas and Zalesski \cite{CZ} that finitely
generated fully residually free groups are conjugacy separable.

 Notice that the decidability of conjugacy problem also follows from the
results of Dahmani and Bumagin. Indeed, Dahmani showed  that $G$ is
relatively hyperbolic and Bumagin proved  that the conjugacy problem
is decidable in relatively hyperbolic groups. We prove that for
finitely generated subgroups $H, K$ of $G$ there are only finitely
many conjugacy classes of intersections $H^g \cap K$ in $G$.
Moreover, one can find a finite set of representatives of these
classes effectively. This implies that one can effectively decide
whether two finitely generated subgroups of $G$ are conjugate or
not, and check if a given finitely generated subgroup is malnormal
in $G$. Observe, that the malnormality problem is decidable in free
groups, but is undecidable in torsion-free hyperbolic groups  -
Bridson and Wise constructed corresponding examples. We provide an
algorithm to find the centralizers of finite sets of elements in
finitely generated fully residually free groups and compute their
ranks. In particular, we prove that for a given finitely generated
fully residually free group $G$ the {\it centralizer spectrum}
$Spec(G) = \{rank(C) \mid C=C_G(g), g \in G\}$, where $rank(C)$ is
the rank of a free abelian group $C$, is finite and one can find it
effectively.

\begin{theorem}\cite{BKM} The isomorphism problem is decidable
in f.g. fully residually free groups.\end{theorem}

We also have an algorithm to solve equations in fully residually
free groups and to construct the abelian JSJ decomposition for them.

Recently Dahmani and Groves \cite{DG} proved
\begin{theorem}The isomorphism problem is decidable
in relatively hyperbolic groups with abelian
parabolics.\end{theorem}

Dahmani \cite{Dah} proved the decidability of the existential theory
of a torsion free relatively hyperbolic group with vitually abelian
parabolic subgroups. This implies our result in \cite{JSJ} about the
decidability of the existential theory of f.g. fully residually free
groups.

\section{Residually free groups}

Any f.g. residually free group can be effectively embedded into a
direct product of a finite number of fully residually free groups
\cite{KMIrc}.

Important steps towards the understanding of the structure of
finitely presented residually free groups were recently made in
\cite{BHMS2007, BHMS2008}.

There exists finitely generated subgroups of $F\times F$ (this group
is residually free but not fully residually free) with unsolvable
conjugacy and word problem (Miller).

In finitely presented residually free groups these problems are
solvable \cite{BHMS2007}.

\begin{theorem}\cite{BHMS2008} Let $G<\Gamma _0\times\ldots\times\Gamma _n$ be the subdirect product of limit groups. Then $G$ is finitely presented
iff it satisfies the virtual surjection to pairs (VSP) property:
$$\forall \ \  0\leq i<j\leq n\ \ |\Gamma _i\times\Gamma _j:P_{ij}(G)|<\infty
.$$\end{theorem}

\bibliographystyle{amsalpha}

\end{document}